\RequirePackage{plautopatch}
\documentclass[11pt]{amsart}

\usepackage[a4paper,margin=28mm]{geometry}
\usepackage{amsmath,amssymb,amsthm,mathtools}
\usepackage{microtype,graphicx,scalerel,xcolor}
\usepackage[pdfusetitle]{hyperref}

\hypersetup{
    colorlinks=true,
}

\usepackage{tikz}
\usetikzlibrary{positioning}
\usetikzlibrary{svg.path}
\definecolor{orcid_color}{HTML}{A6CE39}

\hypersetup{pdfborder={0 0 0}} 
\DeclareRobustCommand{\orcidicon}{%
	\raisebox{.2mm}{\scalerel*{%
	\begin{tikzpicture}[xscale=1,yscale=-1,transform shape]
	\filldraw[color=orcid_color] svg {M256,128c0,70.7-57.3,128-128,128C57.3,256,0,198.7,0,128C0,57.3,57.3,0,128,0C198.7,0,256,57.3,256,128z};
	\filldraw[color=white] svg {M86.3,186.2H70.9V79.1h15.4v48.4V186.2z} svg {M108.9,79.1h41.6c39.6,0,57,28.3,57,53.6c0,27.5-21.5,53.6-56.8,53.6h-41.8V79.1z M124.3,172.4h24.5
		c34.9,0,42.9-26.5,42.9-39.7c0-21.5-13.7-39.7-43.7-39.7h-23.7V172.4z} svg {M88.7,56.8c0,5.5-4.5,10.1-10.1,10.1c-5.6,0-10.1-4.6-10.1-10.1c0-5.6,4.5-10.1,10.1-10.1
		C84.2,46.7,88.7,51.3,88.7,56.8z};
	\end{tikzpicture}}{|}}%
}
\newcommand{\orcid}[1]{\href{https://orcid.org/#1}{\orcidicon}}

\def\ol{\overline}
\def\F{\mathbb F}
\def\R{\mathbb R}
\def\E{\mathbb E}
\def\P{\mathbb P}
\newcommand{\one}{\mathbf 1}
\renewcommand{\AA}{\mathcal A}
\newcommand{\BB}{\mathcal B}
\newcommand{\CC}{\mathcal C}
\newcommand{\DD}{\mathcal D}
\newcommand{\FF}{\mathcal F}
\newcommand{\EE}{\mathcal E}
\newcommand{\GG}{\mathcal G}
\newcommand{\HH}{\mathcal H}
\newcommand{\JJ}{\mathcal J}
\newcommand{\MM}{\mathcal M}
\newcommand{\RR}{\mathcal R}
\newcommand{\UU}{\mathcal U}
\newcommand{\VV}{\mathcal V}
\newcommand{\qbinom}[2]{\genfrac{[}{]}{0pt}{}{#1}{#2}}

\DeclareMathOperator{\Hom}{Hom}
\DeclareMathOperator{\Ball}{B}

\newtheorem{theorem}{Theorem}
\newtheorem{conjecture}[theorem]{Conjecture}
\newtheorem{lemma}[theorem]{Lemma}
\newtheorem{proposition}[theorem]{Proposition}
\newtheorem{claim}[theorem]{Claim}
\newtheorem{corollary}[theorem]{Corollary}
\newtheorem{assertion}[theorem]{Assertion}

\title{A product theorem for $r$-cross intersecting families of subspaces}
\author[Toshihiro Shimizu]{Toshihiro Shimizu\,\orcid{0009-0009-1379-0148}}
\address{Fujitsu Research, Fujitsu Limited, Kawasaki 211-8588, Japan}
\email{t.shimizu\_2@fujitsu.com}
\author[Norihide Tokushige]{Norihide Tokushige\,\orcid{0000-0002-9487-7545}}
\address{College of Education, University of the Ryukyus, Nishihara  903-0213, Japan}
\email{hide@cs.u-ryukyu.ac.jp}
\urladdr{https://n-tokushige.github.io/x/}

\begin{document}
\begin{abstract}
Let $V$ be an $n$-dimensional vector space over a finite field of order $q$.
Let $r\geq 3$, $(r-1)n\geq rk$ and let $\mathcal F_1,\ldots,\mathcal F_r\subset 
\genfrac{[}{]}{0pt}{}{V}{k}$,
where $\genfrac{[}{]}{0pt}{}{V}{k}$ denotes the set of $k$-dimensional subspaces of $V$.
Suppose that
$F_1\cap\cdots\cap F_r\neq\{0\}$ holds for all
$F_i\in\mathcal F_i$, $1\leq i\leq r$. Then we show that
$\prod_{i=1}^r|\mathcal F_i|\leq\genfrac{[}{]}{0pt}{}{n-1}{k-1}$, 
provided $n-k$ is sufficiently large for fixed $q$ and $r$. 
Moreover, equality holds if and only if there is a common line $L$ such that every family 
$\mathcal F_i$ consists of all $k$-dimensional subspaces containing the line $L$.
One of the main tools of the proof is a junta theorem concerning intersecting linear maps
obtained by Ellis, Kindler, and Lifshitz.
\end{abstract}
\maketitle
\tableofcontents

\section{Introduction}
We study $r$-cross intersecting families of subspaces in a vector space over a finite field.
Our main result is an Erd\H{o}s--Ko--Rado type inequality for the product of the sizes of
these families. To this end, we introduce the width of a family, and then we show 
an inequality for the sum of the widths of the corresponding $r$-cross union families.
One of the main tools for this part is a recent junta theorem for linear maps 
obtained by Ellis, Kindler, and Lifshitz \cite{EKL}.

Let $q$ be a fixed prime power, and let $V_n$ denote an $n$-dimensional
vector space over $\F_q$. 
Let $\qbinom {V_n}k$ denote the set of all $k$-dimensional subspaces of $V_n$.
We say that a family $\FF\subset\qbinom{V_n}k$ is $r$-wise intersecting
if $F_1\cap\cdots\cap F_r\neq \{0\}$ for all $F_i\in\FF$, $1\leq i\leq r$. 
For a real number $x\geq k-1$, 
we define the Gaussian binomial coefficient $\qbinom xk$ by
\[
\qbinom xk:=\prod_{j=0}^{k-1} \,\frac{q^{x-j}-1}{q^{k-j}-1}. 
\]
It is continuous and strictly increasing on $[k-1,\infty)$.
Then we have $|\qbinom {V_n}k |=\qbinom nk$.
Chowdhury and Patk\'os \cite{CP} proved the following vector space version of 
the Erd\H{o}s--Ko--Rado theorem. (The classification of the extremal structures for the
case $r=2$ and $n=2k$ is due to Newman \cite{Newman}, see also
Godsil--Newman \cite{Godsil-Newman}.)
\begin{theorem}\label{thm:1}
Let $n,k,r$ be positive integers with $r\geq 2$ and $(r-1)n\geq rk$.
If $\FF\subset\qbinom{V_n}k$ is an $r$-wise intersecting family, then
\[
|\FF|\leq\qbinom{n-1}{k-1}. 
\]
Moreover, equality holds if and only if $\FF=\{F\in\qbinom{V_n}k:L\leq F\}$ for some line $L\in\qbinom{V_n}1$, or $r=2$, $n=2k$, and $\FF=\qbinom{H}k$ for some $H\in\qbinom{V_n}{n-1}$.
\end{theorem}

We say that families
$\FF_1,\ldots,\FF_r\subset\qbinom{V_n}k$ are $r$-cross intersecting
if $F_1\cap\cdots\cap F_r\neq \{0\}$ for all $F_i\in\FF_i$, $1\leq i\leq r$. 
A natural conjecture is the following.
\begin{conjecture}\label{conj:2}
Let $n,k,r$ be positive integers with $r\geq 2$ and $(r-1)n\geq rk$.
If $\FF_1,\ldots,\FF_r\subset\qbinom{V_n}k$ are $r$-cross intersecting
families, then
\[
\prod_{i=1}^r|\FF_i|\leq\qbinom{n-1}{k-1}^r.
\]
Moreover, equality holds if and only if $\FF_1=\cdots=\FF_r=\{F\in\qbinom{V_n}k:L\leq F\}$ for some line $L\in\qbinom{V_n}1$, 
or $r=2$, $n=2k$, and $\FF_1=\FF_2=\qbinom{H}k$ for some $H\in\qbinom{V_n}{n-1}$.
\end{conjecture}
Indeed, the conjecture is true if $r=2$ even in the following stronger form.

\begin{theorem}[\cite{Tokushige}]\label{thm:3}
Let $k\geq t\geq 1$ and $n\geq 2k$. Suppose that two families 
$\AA,\BB\subset\qbinom{V_n}k$ satisfy $\dim(A\cap B)\geq t$ for all $A\in\AA$
and $B\in\BB$. Then
\[
|\AA||\BB|\leq\qbinom{n-t}{k-t}^2. 
\]
Moreover, equality holds if and only if 
$\AA=\BB=\{F\in\qbinom{V_n}k:T\leq F\}$ for some $T\in\qbinom{V_n}t$,
or $n=2k$, and $\AA=\BB=\qbinom{H}k$ for some $H\in\qbinom{V_n}{n-t}$.
\end{theorem}
\noindent

We say that families $\FF_1,\ldots,\FF_r\subset\qbinom{V_n}k$ are $r$-cross $t$-intersecting if
$\dim(F_1\cap\cdots\cap F_r)\geq t$ for all $F_i\in\FF_i$, $1\leq i\leq r$ (we omit
$t$ if $t=1$). 
If one of the families is empty, then $\prod_{i=1}^r|\FF_i|=0$. Otherwise,
every pair of families is $2$-cross $t$-intersecting; hence Theorem~\ref{thm:3}
implies $\prod_{i=1}^r|\FF_i|=\left((|\FF_1||\FF_2|)(|\FF_2||\FF_3|)\cdots(|\FF_r||\FF_1|)\right)^{\frac12}\leq\qbinom{n-t}{k-t}^r$
provided $k\geq t\geq 1$ and $n\geq 2k$.
Thus the challenge in Conjecture~\ref{conj:2} is for the case $r\geq 3$ and $n<2k$.

In this paper, we verify the conjecture partially as follows.

\begin{theorem}\label{thm:4}
Let $r\geq 3$ and let $n-k$ be a sufficiently large positive integer 
depending on $q$ and $r$ only.
If $(r-1)n\geq rk$, and 
$\FF_1,\ldots,\FF_r\subset\qbinom{V_n}k$ are $r$-cross intersecting
families, then
\[
\prod_{i=1}^r|\FF_i|\leq\qbinom{n-1}{k-1}^r.
\]
Moreover, equality holds if and only if $\FF_1=\cdots=\FF_r=\{F\in\qbinom{V_n}k:L\leq F\}$ for some line $L\in\qbinom{V_n}1$.
\end{theorem}

We say that $r$ families $\FF_1,\ldots,\FF_r\subset\qbinom{V_n}l$ of subspaces are $r$-cross union if $F_1+F_2+\cdots+F_r\neq V_n$ for all 
$F_i\in\FF_i$ with $1\leq i\leq r$.
Fix a nondegenerate bilinear form on $V_n$, e.g., the standard dot product.
For $\FF\subset\qbinom{V_n}l$ we define its orthogonal-complement family 
by $\FF^\perp:=\{F^\perp:F\in\FF\}\subset\qbinom{V_n}{n-l}$.
Since $(F_1+\cdots+ F_r)^\perp=F_1^\perp\cap\cdots\cap F_r^\perp$ for all $F_i\in\FF_i$, it follows that 
$\FF_1,\ldots,\FF_r\subset\qbinom{V_n}l$ are $r$-cross union if and only if
$\FF_1^\perp,\ldots,\FF_r^\perp\subset\qbinom{V_n}{n-l}$ are $r$-cross intersecting.
(This property will be used to check that Theorem~\ref{thm:4} and 
Theorem~\ref{thm:6} are equivalent.)

For a non-empty family $\FF\subset\qbinom{V_n}l$, we have
$1\leq|\FF|\leq\qbinom nl$.
In this case, we define its width, denoted by $\|\FF\|_l$,
as a real number $x\in[l,n]$ satisfying $|\FF|=\qbinom xl$.
This is well-defined because $x\mapsto\qbinom xl$ is continuous and
strictly increasing on $[l,n]$.
For the empty case, we define $\|\emptyset\|_l=0$.
Notice that if $W\leq V_n$ is a subspace with $\dim W\geq l$
and $\FF=\qbinom{W}l$, then $\|\FF\|_l=\dim W$. 

Our main result is the following.

\begin{theorem}\label{thm:5}
Let $r\geq 3$ and let $l\geq r$ be a sufficiently large positive integer 
depending on $q$ and $r$ only.
If $l+1\leq n\leq rl$, and $\FF_1,\FF_2,\ldots,\FF_r\subset\qbinom{V_n}l$
are $r$-cross union families, then
\[
  \sum_{i=1}^r\|\FF_i\|_l\leq r(n-1).
\]
Moreover, equality holds if and only if 
$\FF_1=\FF_2=\cdots=\FF_r=\qbinom Hl$ for some hyperplane $H\in\qbinom{V_n}{n-1}$.
\end{theorem}

If one of the families is empty, then $\sum_{i=1}^r\|\FF_i\|_l\leq (r-1)n< r(n-1)$
because $r\leq l<n$. Thus, Theorem~\ref{thm:5} is trivial if 
one of the families is empty. (The same applies to Theorem~\ref{thm:7} below.) 

By Theorem~\ref{thm:5} with the inequality of arithmetic and geometric means, 
the following result easily follows as we will see in the next section.

\begin{theorem}\label{thm:6}
Let $r\geq 3$ and let $l\geq r$ be a sufficiently large positive integer depending on
$q$ and $r$ only.
If $l+1\leq n\leq rl$ and $\FF_1,\FF_2,\ldots,\FF_r\subset\qbinom{V_n}l$ are $r$-cross union
families, then 
\[
\prod_{i=1}^r|\FF_i|\leq\qbinom{n-1}{l}^r. 
\]
Moreover, equality holds if and only if 
$\FF_1=\FF_2=\cdots=\FF_r=\qbinom Hl$ for some hyperplane $H\in\qbinom{V_n}{n-1}$.
\end{theorem}
Theorem~\ref{thm:6} is equivalent to Theorem~\ref{thm:4}. 
To see this, suppose that $l=n-k$. Then $n\leq rl$ is equivalent to $(r-1)n\geq rk$,
and $\qbinom{n-1}l=\qbinom{n-1}{k-1}$.

As for the case $r=2$, we have the following result. Note that there is an
additional extremal configuration only for the case $r=2$ and $n=2l$.

\begin{theorem}\label{thm:7}
Let $l\geq 2$ be an integer. 
If $l+1\leq n\leq 2l$, and $\FF_1,\FF_2\subset\qbinom{V_n}l$
are $2$-cross union families, then
\[
\|\FF_1\|_l+\|\FF_2\|_l\leq 2(n-1).
\]
Moreover, equality holds if and only if 
$\FF_1=\FF_2=\qbinom Hl$ for some hyperplane $H\in\qbinom{V_n}{n-1}$,
or $n=2l$ and $\FF_1=\FF_2=\{F\in\qbinom{V_{2l}}l:L\leq F\}$ 
for some line $L\in\qbinom{V_{2l}}{1}$.
\end{theorem}

Finally we mention some related results. In this paper we only consider the case
where all families have the same uniformity (the same subspace dimension).
One can also consider the case where families have different uniformities.
The first result in this direction was obtained by Suda and Tanaka \cite{Suda-Tanaka}. They proved that if $n\geq 2k_1\geq 2k_2$, and $\FF_1\subset\qbinom{V_n}{k_1}$
and $\FF_2\subset\qbinom{V_n}{k_2}$ are 2-cross intersecting, then
$|\FF_1||\FF_2|\leq\qbinom{n-1}{k_1-1}\qbinom{n-1}{k_2-1}$.
Recently, Cao et al.\ \cite{Cao} proved that if $r\geq 2$, 
$k_1\geq k_2\geq\cdots\geq k_r\geq t\geq 1$,
$n\geq k_1+k_2+t+1$, and $\FF_i\subset\qbinom{V_n}{k_i}$ ($1\leq i\leq r$) are
$r$-cross $t$-intersecting, then $\prod_{i=1}^r|\FF_i|\leq\prod_{i=1}^r\qbinom{n-t}{k_i-t}$. Then, improving this result, Wen and Lv \cite{Wen-Lv} showed that 
the lower bound on $n$ can be replaced with weaker conditions
$q=2$ and $n\geq k_1+k_r+2$, or $q\geq 3$ and $n\geq k_1+k_r+1$.
Note that if $k_1=k_r=:k$, then it is easy to show 
$\prod_{i=1}^r|\FF_i|\leq\qbinom{n-t}{k-t}^r$ for $n\geq 2k$,
and if moreover $t=1$, then the conjectured lower bound on $n$ is $n\geq rk/(r-1)$.

Another related direction is to consider the sum of the sizes instead of the
product of sizes of families. Indeed, if $\FF_1,\ldots,\FF_r\subset\qbinom{V_n}k$
are $r$-cross $t$-intersecting with $\sum_{i=1}^r|\FF_i|\leq r\qbinom{n-t}{k-t}$,
then $\prod_{i=1}^r|\FF_i|\leq \qbinom{n-t}{k-t}^r$ easily follows from the
inequality of arithmetic and geometric means (AM-GM inequality).
However, $\sum_{i=1}^r|\FF_i|$ can exceed $r\qbinom{n-t}{k-t}$ even if
$r=2$ and $n\geq 2k$, see Wang--Zhang \cite{Wang-Zhang}.
This is one of the reasons we introduced the width of a family, and 
Theorem~\ref{thm:7} is new in this sense, though the corresponding product
result was already known.

\section{Outline and two reductions}
Our main result, Theorem~\ref{thm:5} and its application
Theorem~\ref{thm:6} will be obtained through the following sequence.
\begin{center}
 Proposition~\ref{prop:20}
$\Rightarrow$
 Proposition~\ref{prop:10}
$\Rightarrow$
 Proposition~\ref{prop:11}
$\Rightarrow$
 Theorem~\ref{thm:5}
$\Rightarrow$
 Theorem~\ref{thm:6}.
\end{center}
After preparing some tools for proofs in the next section, the three propositions 
above will be proved in the last section.
In this section, we first prove Theorem~\ref{thm:6} using Theorem~\ref{thm:5}. 
We then prove Theorem~\ref{thm:5} using Proposition~\ref{prop:11}
by induction on $rl-n$. 

Proposition~\ref{prop:11} is the base case ($rl-n=0$) of the
induction. This proposition, in turn, is proved by induction on $r$, and its base case 
($r=3$) is Proposition~\ref{prop:10}.
In this sense, Proposition~\ref{prop:10} is the core of the proof. 
To prove Proposition~\ref{prop:10},
we use Proposition~\ref{prop:20}, which is based on a junta result for linear maps 
(Lemma~\ref{lem:18}) due to Ellis, Kindler, and Lifshitz \cite{EKL}.

For Theorem~\ref{thm:7} ($r=2$), we give a direct and much simpler
proof in the last section.

\subsection{Reduction of Theorem~\ref{thm:6} to Theorem~\ref{thm:5}}
\begin{proof}
We prove Theorem~\ref{thm:6} under the assumption of Theorem~\ref{thm:5}.
If one of the families $\FF_i$ is empty, then $\prod_i|\FF_i|=0$ and 
Theorem~\ref{thm:6} clearly holds.
So, we may assume that all $\FF_i$ are non-empty.
Let $x_i=\|\FF_i\|_l$, that is, $|\FF_i|=\qbinom{x_i}l$,
for $1\leq i\leq r$. Then $x_i\in[l,n]$.
We need the following inequality, which is a consequence of the AM-GM inequality.

\begin{claim}\label{claim:8}
 Let $b_1,\ldots,b_r$ be non-negative reals. Then, we have
\[
(b_1+1)\cdots(b_r+1)\geq\{(b_1\cdots b_r)^{\frac1r}+1\}^r.
\]
\end{claim}

\begin{proof}
For each $1\leq i\leq r-1$, it follows from the AM-GM inequality that
\[
\sum_{I\in\binom{[r]}i}\prod_{j\in I}b_j
\geq \binom ri 
\left(\prod_{I\in\binom{[r]}i}\prod_{j\in I}b_j\right)^\frac1{\binom ri}
=\binom ri (b_1\cdots b_r)^\frac ir.
\]
Note that the cases $i=0$ and $i=r$ are identities.
Thus we have
\begin{align*}
 (b_1+1)\cdots(b_r+1)
&=\sum_{i=0}^r\sum_{I\in\binom{[r]}i}\prod_{j\in I}b_j\\
&\geq \sum_{i=0}^r \binom ri (b_1\cdots b_r)^\frac ir 
=\{(b_1\cdots b_r)^{\frac1r}+1\}^r,
\end{align*}
as desired.
\end{proof}

\begin{claim}\label{claim:9}
\[
\qbinom{x_1}l\cdots\qbinom{x_r}l\leq
\qbinom{\frac{x_1+\cdots+x_r}r}{l} ^r.
\]
\end{claim}
\begin{proof}
Let $A=\prod_{j=0}^{l-1}({q^{l-j}-1})$.
Then we see that
\begin{align*}
\text{LHS}= & A^{-r}
\prod_{j=0}^{l-1}(q^{x_1-j}-1)\cdots(q^{x_r-j}-1),\\
\text{RHS}= &
A^{-r}
\prod_{j=0}^{l-1}\left(q^{\frac{x_1+\cdots+x_r}r-j}-1\right)^r
=A^{-r}\prod_{j=0}^{l-1}\{(q^{x_1-j}\cdots q^{x_r-j})^\frac 1r-1\}^r.
\end{align*}
Thus it suffices to show
\[
(q^{x_1-j}-1)\cdots(q^{x_r-j}-1)\leq\{(q^{x_1-j}\cdots q^{x_r-j})^\frac 1r-1\}^r 
\]
for each $0\leq j<l$. By letting $b_i=q^{x_i-j}-1$, we can rewrite the
above inequality as
\[
b_1\cdots b_r \leq \{( (b_1+1)\cdots(b_r+1))^\frac 1r -1\}^r, 
\]
or equivalently,
\[
( (b_1+1)\cdots(b_r+1))^\frac 1r \geq (b_1\cdots b_r)^\frac 1r +1, 
\]
which follows from Claim~\ref{claim:8}.
\end{proof}

By Claim~\ref{claim:9}, we have
\[
\prod_{i=1}^r|\FF_i|=\qbinom{x_1}l\cdots\qbinom{x_r}l\leq
\qbinom{\frac{x_1+\cdots+x_r}r}{l} ^r.
\]
Then Theorem~\ref{thm:5} gives $\frac{x_1+\cdots+x_r}r\leq n-1$ and
\[
\prod_{i=1}^r|\FF_i|\leq\qbinom{\frac{x_1+\cdots+x_r}r}{l}^r\leq \qbinom{n-1}{l}^r. 
\]
Moreover, if $\prod_{i=1}^r|\FF_i|=\qbinom{n-1}{l}^r$, then $x_1+\ldots+x_r=r(n-1)$.
Thus, by Theorem~\ref{thm:5}, there is a hyperplane $H\in\qbinom{V_n}{n-1}$
such that $\FF_1=\cdots=\FF_r=\qbinom Hl$.
This completes the proof of Theorem~\ref{thm:6} assuming Theorem~\ref{thm:5}.
\end{proof}

\subsection{Reduction of Theorem~\ref{thm:5} to Proposition~\ref{prop:11}}

We start with the case when $r=3$ and $n=3l$, which is the hardest part.

\begin{proposition}\label{prop:10}
Let $l$ be a sufficiently large positive integer.
Suppose that three families $\FF_1,\FF_2,\FF_3\subset\qbinom{V_{3l}}l$
are $3$-cross union.
Then, we have
\[
  \|\FF_1\|_l+\|\FF_2\|_l+\|\FF_3\|_l\leq 3(3l-1).
\]
Moreover, equality holds if and only if 
$\FF_1=\FF_2=\FF_3=\qbinom Hl$ for some hyperplane $H\in\qbinom{V_{3l}}{3l-1}$.
\end{proposition}

Then, by induction on $r$, we show the case when $n=rl$ for all $r\geq 3$.

\begin{proposition}\label{prop:11}
Let $r\geq 3$, and let $l$ be a sufficiently large positive integer.
Suppose that $r$ families $\FF_1,\ldots,\FF_r\subset\qbinom{V_{rl}}l$
are $r$-cross union.
Then, we have
\[
  \|\FF_1\|_l+\cdots+\|\FF_r\|_l\leq r(rl-1).
\]
Moreover, equality holds if and only if 
$\FF_1=\cdots=\FF_r=\qbinom Hl$ for some hyperplane $H\in\qbinom{V_{rl}}{rl-1}$.
\end{proposition}

Finally, we prove Theorem~\ref{thm:5} by induction on $s$, where 
$s=rl -n$. The initial step $s=0$ is exactly Proposition~\ref{prop:11}.
Then, the induction step is easy, and we include it here.

\begin{proof}[Proof of Theorem~\ref{thm:5} from Proposition~\ref{prop:11}]
Let $s\geq 0$. 
Suppose that Theorem~\ref{thm:5} is true for the case $rl -n=s$, 
and we will consider the case $rl -n=s+1$. Write $V$ for $V_n$.

Let $\FF_1,\ldots,\FF_r\subset\qbinom{V}l$ be $r$-cross union families
with $rl -n=s+1$. Recall that $x_i=\|\FF_i\|_l$ for $1\leq i\leq r$,
and we shall show that 
\begin{equation}\label{eq:1}
\sum_{i=1}^rx_i\leq r(n-1). 
\end{equation}
If one of the families $\FF_i$ is empty, then 
$\sum_{i=1}^rx_i\leq (r-1)n\leq r(n-1)$ because $r\leq l+1\leq n$.
Hence we may assume that all the families are non-empty.

Let $W$ be an $(n+1)$-dimensional vector space with orthogonal decomposition
$W=V\oplus L$ where $\dim L=1$. Let $p:W\to V$ be the projection.
Let 
\[
\HH_i=\{H\in\qbinom Wl : p(H)\leq F\text{ for some }F\in\FF_i\}, 
\]
and let $\HH_i=\AA_i\cup\BB_i$ be a partition defined by
\begin{align*}
\AA_i&=\{H\in\HH_i:p(H)\in \FF_i\},\\
\BB_i&=\{H\in\HH_i:p(H)\in \Delta_{l-1}(\FF_i)\},
\end{align*}
where $\Delta_j(\FF_i)=\{J\in\qbinom{V}j:J\leq F\text{ for some }F\in\FF_i\}$
is the $j$-th shadow of $\FF_i$.
If we represent $\FF_i$ by $l\times n$ matrices in reduced echelon form,
then we can represent $\HH_i$ by $l\times (n+1)$ matrices in reduced echelon 
form as well. Namely, a matrix of $\AA_i$ is obtained from an $l\times n$
matrix of $\FF_i$ by adding any column vector in $\F_q^l$ as an $(n+1)$-th
column, and a matrix of $\BB_i$ is obtained from an $(l-1)\times n$ matrix
of $\Delta_{l-1}(\FF_i)$ by adding all zeros in $l$-th row and
$(n+1)$-th column except $(l,n+1)$-position in which we put $1$.
By the construction we have $|\AA_i|=q^l|\FF_i|$. 
(We will give alternative reasoning in \S~\ref{sec:lifts}.)
As for $\BB_i$, we invoke a vector space version of the Kruskal--Katona theorem 
(\cite{CP}, Theorem 1.4) due to Chowdhury and Patk\'os, which states that
if $|\FF|=\qbinom xl$, then $|\Delta_{l-1}(\FF)|\geq\qbinom x{l-1}$.
Then we have $|\BB_i|=|\Delta_{l-1}(\FF_i)|\geq\qbinom{x_i}{l-1}$, and
\[
|\HH_i|=|\AA_i|+|\BB_i|\geq q^l\qbinom{x_i}l+\qbinom{x_i}{l-1}
=\qbinom{x_i+1}l,
\]
where the last equality follows from the definition of $q$-binomial coefficient
(see also \S~\ref{sec:q-binom}). Therefore, we obtain
\begin{equation}\label{eq:2}
 z_i:=\|\HH_i\|_l\geq x_i+1.
\end{equation}

Now we notice that $\HH_1,\ldots,\HH_r\subset\qbinom{W}l$ are $r$-cross
union families, in fact, their sum cannot equal $W$.
Indeed, if $H_1+\cdots+H_r=W$, then $V=p(W)=p(H_1)+\cdots+p(H_r)\leq F_1+\cdots+F_r$, contradicting the $r$-cross union condition.
Moreover, since $rl-(n+1)=s$, we can apply induction hypothesis to 
$\HH_1,\ldots,\HH_r$, and we get
\begin{equation}\label{eq:3}
 \sum_{i=1}^r z_i\leq r((n+1)-1)=rn.
\end{equation}
Then, \eqref{eq:1} follows from \eqref{eq:2} and \eqref{eq:3}.
(We remark that this part of the proof works for $r=2$ as well, and we will use it
to prove Theorem~\ref{thm:7}.)

Finally we consider the structure of the families for the equality case.
Suppose that equality holds in \eqref{eq:1}. Then equality holds both in
\eqref{eq:2} and \eqref{eq:3}.
By induction hypothesis, there is a hyperplane $H\in\qbinom{W}n$ such that
$\HH_i=\qbinom{H}l$ for all $i$. 
Note that if $F\in\qbinom Vl$, then $F\in\HH_i$ if and only if $F\in\FF_i$.

If $H=V$, then $\HH_i=\qbinom Vl$, and every $F\in\qbinom Vl$ satisfies
$p(F)=F$ and so $F\in\FF_i$ for all $i$. 
Since $n=rl-s-1<rl$, we can partition a basis of $V$ into $r$ groups of size 
at most $l$, and extend the span of each group to a subspace $F_i\in\FF_i$ 
($1\leq i\leq r$) so that $F_1+\cdots+F_r=V$. This contradicts the $r$-cross union
condition. Thus, we may assume that $H\not\subset V$, and let $H':=H\cap V$. 
Then $H'\in\qbinom {V}{n-1}$ is a hyperplane in $V$.
If $F\in\FF_i$, then $F\leq H$ because $\FF_i\subset\HH_i$, and
$F\leq V$ because $\FF_i\subset\qbinom Vl$.
Thus $F\leq H\cap V$ and $\FF_i\subset\qbinom {H'}l$.
On the other hand, if $F\in\qbinom{H'}l$, then $F\in\HH_i$ because
$\qbinom{H'}l\subset\qbinom Hl=\HH_i$, and $p(F)=F$. Thus $F\in\FF_i$ and $\qbinom{H'}l\subset\FF_i$. Therefore, we have $\qbinom{H'}l=\FF_i$ for all $i$, as desired.
This completes the proof of Theorem~\ref{thm:5} assuming Proposition~\ref{prop:11}.
\end{proof}

We extract a result from the proof above, and record it for later use to prove 
Theorem~\ref{thm:7}. Note that we did not use the assumption $r\geq 3$ in the
proof of this part, in fact, the following holds for $r=2$ as well.

\begin{lemma}\label{lem:12}
Let $r\geq 2$ and $l+1\leq n<rl$.
Suppose that $\FF_1,\ldots,\FF_r\subset\qbinom{V_n}l$ are non-empty $r$-cross union
families. Then there exist $r$-cross union families
$\HH_1,\ldots,\HH_r\subset\qbinom{V_{n+1}}l$ such that
$\|\HH_i\|_l\geq\|\FF_i\|_l+1$ for all $1\leq i\leq r$.
Moreover, if there is a hyperplane $H_n\leq V_{n+1}$ such that 
$\HH_i=\qbinom{H_n}l$ for all $i$, then there is a hyperplane 
$H_{n-1}\leq V_n$ such that $\FF_i=\qbinom{H_{n-1}}l$ for all $i$.
\end{lemma}

\section{Preliminaries}
In this paper, we always assume that $q$ is a fixed prime power.
For a constant depending on $q$, we write $O(1)$ instead of $O_q(1)$.
(We use $O_q(1)$ in Appendix only.) We also write $O(1/l)$ for $O_q(1/l)$
when $l\to\infty$.
For a positive integer $n$, let $[n]:=\{1,2,\ldots,n\}$.

\subsection{\texorpdfstring{$q$-Binomial}{q-Binomial} coefficients}\label{sec:q-binom}
By definition of the $q$-binomial coefficients, it follows that
$\qbinom {x+1}l=q^l\qbinom{x}l+\qbinom{x}{l-1}$
for every positive integer $l$ and every real number $x\geq l-1$,
and $\qbinom nl=\qbinom n{n-l}$ for every integer $0\leq l\leq n$.
For the situation $l\to\infty$, we often use the following estimates without mention.

\begin{lemma}
Let $a>1$, and let $(b_l)$ be a bounded sequence of real numbers.
Suppose that $x_l\geq l$ for all sufficiently large $l$.
Then $x_l=al+b_l+O(1/l)$ if and only if $\qbinom{x_l}l=q^{(a-1)l^2+b_ll+O(1)}$.
All implicit constants may be chosen uniformly when $(b_l)$ and
$l(x_l-al-b_l)$ range over fixed bounded sets.
\end{lemma}

\begin{proof}
Set $c_q:=\prod_{r=1}^\infty(1-q^{-r})\in(0,1)$. For every real $x\geq l$,
\[
 \qbinom xl=q^{l(x-l)}\frac{\prod_{r=1}^l(1-q^{-(x-l+r)})}{\prod_{r=1}^l(1-q^{-r})}.
\]
Since $x\geq l$, the quotient of products lies between $1$ and $c_q^{-1}$.  Hence
\[
 q^{l(x-l)}\leq \qbinom xl\le c_q^{-1}q^{l(x-l)},
\]
which gives both implications immediately.
\end{proof}

\subsection{A quotient map and the number of lifts}\label{sec:lifts}
Let $V$ be an $n$-dimensional vector space over $\F_q$, and let $W$
be a $k$-subspace of $V$.
The number of $i$-subspaces whose intersection with $W$ has dimension $j$
is 
\begin{align}\label{eq:4}
q^{(k-j)(i-j)}\qbinom{n-k}{i-j}\qbinom kj  
\end{align}
(see, e.g., Lemma~9.3.2 in \cite{Godsil-Meagher}).

Suppose that $n\geq k+l$. Let $\overline V=V/W$ be the
quotient space, and $\pi:V\to\overline V$ be the quotient map.
For $\overline F\in\qbinom{\overline V}l$, we say that $F\in\qbinom Vl$
is a lift of $\overline F$ if $\pi(F)=\overline F$. Let $U=\pi^{-1}(\overline F)$. 
Choose one lift $F_0$ of $\ol F$. Then $U=F_0\oplus W$.
Every lift of $\ol F$ is the graph $F_\phi=\{\phi(x)+x:x\in F_0\}$
of a unique linear map $\phi:F_0\to W$. Hence the number of lifts is 
\begin{align}\label{eq:5}
|\Hom(F_0,W)| = q^{kl}=q^{\dim W\cdot\,\dim\ol F}. 
\end{align}

\subsection{The number of linear maps with a constraint}
Let $A,B$ be $l$-dimensional vector spaces over $\F_q$, and let $A^*,B^*$
be the corresponding dual spaces. Let $\Hom(B,A)$ denote the set of linear
maps from $B$ to $A$.
We often identify $\Hom(B,A)$ with the set of $\dim A\times\dim B$ matrices
over $\F_q$ (by fixing bases of $B$ and $A$).

Let $0\neq u\in B$ and $v\in A$ be given. Then we have
\begin{align}\label{eq:6}
|\{M\in\Hom(B,A):Mu=v\}|=q^{l(l-1)}. 
\end{align}
To see this, fix a basis $\{u,e_2,e_3,\ldots,e_l\}$ of $B$. Then
$M:B\to A$ is determined by choosing the image of the basis. It is fixed
that $u$ goes to $v$, and the other $e_j$ can go any vector in $A$.
Thus, there are $q^l$ choices for each $e_j$, and so $q^{l(l-1)}$ choices in total. 

Recall that $M^*\in\Hom(A^*,B^*)$ is defined by $M^*(\phi)=\phi\circ M$,
and $M\mapsto M^*$ gives $\Hom(B,A)\cong\Hom(A^*,B^*)$. By duality,
for fixed $0\neq \phi\in A^*$ and $\psi\in B^*$, we also have
\begin{align}\label{eq:7}
 |\{M\in\Hom(B,A):\phi\circ M=\psi\}|=
 |\{M^*\in\Hom(A^*,B^*):M^*(\phi)=\psi\}|=q^{l(l-1)}.
\end{align}

\subsection{\texorpdfstring{$q$-Kneser}{q-Kneser} graphs}
Let $n\geq 2k$, and 
let $V$ denote an $n$-dimensional vector space over $\F_q$, and let 
$\Omega=\qbinom Vk$. The $q$-Kneser $K_q(n,k)$ on the vertex set $\Omega$ is
defined as follows.
Two vertices $A,B\in\Omega$ are adjacent, denoted by $A\sim B$, if and only if
$A\cap B=\{0\}$. The eigenvalues of $K_q(n,k)$ are 
\begin{align}\label{eq:8}
 \theta_i=(-1)^iq^{\binom i2+k(k-i)}\qbinom{n-k-i}{k-i}, 
\end{align}
where $i=0,1,\ldots,k$, with multiplicities $m_i:=\qbinom ni-\qbinom n{i-1}$
(see, e.g., Corollary~9.7.1 in \cite{Godsil-Meagher}).
In particular, $\theta_0=q^{k^2}\qbinom{n-k}k$, $m_0=1$, and 
$\theta_0\geq-\theta_1\geq\theta_2\geq-\theta_3\geq\cdots$.
For $0\leq i\leq k$, let $\{v_{i,j}:j\in[m_i]\}$ be an orthonormal basis for the
eigenspace corresponding to $\theta_i$, 
and let $T$ be the adjacency matrix of $K_q(n,k)$. Then $Tv_{i,j}=\theta_iv_{i,j}$,
and $v_{0,1}=\one/\sqrt{|\Omega|}$, where $\one$ is the all ones vector. 

For a function $f:\Omega\to[0,1]$, define $\E f:=|\Omega|^{-1}\sum_{F\in\Omega}f(F)$.
For $\mu:=\E f$ we have $\langle f,\one\rangle=\mu|\Omega|$, and so 
we can write $f=\mu\one+f_*$,
where $\mu=\E f$ and $f_*$ is perpendicular to $\one$. (Here
$\langle\cdot,\cdot\rangle$ denotes the standard inner product, and we define
$|f|_2:=\sqrt{\langle f,f\rangle}$.)
By expanding $f_*=\sum_{i=1}^k\sum_{j\in[m_i]} \lambda_{i,j}v_{i,j}$, we have
\begin{align}\label{eq:9}
|Tf_*|_2^2=\langle Tf_*,Tf_*\rangle =\sum_{i=1}^k\sum_{j\in[m_i]}\theta_i^2\lambda_{i,j}^2\leq\theta_1^2\sum_{i=1}^k\sum_{j\in[m_i]}\lambda_{i,j}^2=\theta_1^2\,|f_*|_2^2.
\end{align}
We also have
\begin{align*}
|f_*|_2^2&=|f-\mu\one|_2^2=\sum_{F\in\Omega}(f(F)-\mu)^2
=\sum_{F\in\Omega}f^2(F)-2\mu\sum_{F\in\Omega}f(F)+\mu^2|\Omega|\\
&=\sum_{F\in\Omega}f^2(F)-\mu^2|\Omega|\leq\sum_{F\in\Omega}f^2(F)
\leq\sum_{F\in\Omega}f(F)=\mu|\Omega|.
\end{align*}
Therefore, we have
\begin{align}\label{eq:10}
 \sum_{F\in\Omega}\left((Tf)(F)-\theta_0\mu\right)^2
=|T(f-\mu\one)|_2^2=|Tf_*|_2^2\leq\theta_1^2|f_*|_2^2\leq\theta_1^2\mu|\Omega|.
\end{align}

\subsubsection{$K_q(ml,l)$}\label{sec:3-d q-K}
Consider the case $n=ml$, $k=l$.
Let $\mu:=\E\one_\AA$, where $\one_\AA$ is the indicator vector of $\AA$.
Let $d_\AA:=T\one_\AA$. Then $d_\AA(F)=(T\one_\AA)(F)=|\{A\in\AA:A\cap F=\{0\}\}|$.
It follows from \eqref{eq:8} that $\theta_0=q^{l^2}\qbinom{(m-1)l}l$ and 
$\frac{|\theta_j|}{\theta_0}\leq\frac{|\theta_1|}{\theta_0}=q^{-l}\qbinom{(m-1)l-1}{l-1}\qbinom{(m-1)l}l^{-1}=q^{-(m-1)l+O_m(1)}$.

\begin{lemma}
\label{lem:14}
Let $\emptyset\neq\AA\subset\Omega$ and $\mu=\E\one_\AA$.
Then, for every fixed $0<\eta<1$, we have
\[
|\EE|  \leq  \eta^{-2}\mu^{-1}q^{-2(m-1)l+O_m(1)}|\Omega|, 
\]  
where $\EE:=\{F\in\Omega:|d_\AA(F)-\theta_0\mu|>\eta\theta_0\mu\}
\supset\{F\in\Omega:d_\AA(F)<(1-\eta)\theta_0\mu\}$. 
\end{lemma}

\begin{proof}
Apply \eqref{eq:10} with $f=\one_\AA$. Then we have
\[
\sum_{F\in\Omega}(d_\AA(F)-\theta_0\mu)^2\leq\theta_1^2\mu|\Omega|. 
\]
Since each $F\in\EE$ satisfies $|d_\AA(F)-\theta_0\mu|^2>(\eta\theta_0\mu)^2$, it follows that 
\[
|\EE|(\eta\theta_0\mu)^2<\sum_{F\in\EE}|d_\AA(F)-\theta_0\mu|^2\leq\theta_1^2\mu|\Omega|,  
\]
which yields the desired upper bound for $|\EE|$.
\end{proof}

\subsubsection{$K_q(2l,l)$}
Consider the case $n=2l$ and $k=l$.
In this case, we have
$\theta_0=q^{l^2}$ and $\theta_1=-q^{l(l-1)}$.
For a function $h:\Omega\to[0,1]$ and a fixed $A_0\in\Omega$, let $\E[h\mid F\cap A_0=\{0\}]=
q^{-l^2}\sum h(F)$, where the sum is taken over all $F\in\Omega$ with $F\cap A_0=\{0\}$.
Then $(Th)(A_0)=\sum_{F\cap A_0=\{0\}}h(F)$ and 
$\E[h\mid F\cap A_0=\{0\}]={\theta}_0^{-1}(Th)(A_0)$.

\begin{lemma}
\label{lem:15}
Let $K>0$ be a constant.
If $\mu=\E h\geq q^{-l-K}$, then
$|\EE_h|\leq q^{-l+O_K(1)}|\Omega|$, where
$\EE_h:=\{A_0\in\Omega:\E[h\mid F\cap A_0=\{0\}]<\frac12\mu\}$.

In particular, if two functions $h_1,h_2:\Omega\to[0,1]$ satisfy 
$\E h_1,\E h_2\geq q^{-l-K}$, then there is some 
$A_0\in\Omega\setminus(\EE_{h_1}\cup\EE_{h_2})$, that is, 
$\E[h_i\mid F\cap A_0=\{0\}]\geq \frac12\E h_i$ for both $i=1,2$.
\end{lemma}

\begin{proof}
Apply \eqref{eq:10} with $f=h$. Then we have
$\sum_{A_0\in\Omega}\left((Th)(A_0)-\theta_0\mu\right)^2 \leq\theta_1^2\mu|\Omega|$.

Each $A_0\in\EE_h$ satisfies $\theta_0^{-1}(Th)(A_0)<\frac12\mu$,
that is, $(Th)(A_0)-\theta_0\mu<-\frac12\mu\theta_0$. Thus, 
\[
\left(\tfrac12\theta_0\mu\right)^2|\EE_h|<
\sum_{A_0\in\Omega}((Th)(A_0)-\theta_0\mu)^2\leq\theta_1^2\mu|\Omega|,
\]
and so $|\EE_h|<\frac{4\theta_1^2|\Omega|}{\theta_0^2\mu}$.
Using $\mu\geq q^{-l-K}$ and $\theta_1/\theta_0=q^{-l}$, 
we get $|\EE_h|<q^{-l+O_K(1)}|\Omega|$.

Since $\E h_i\geq q^{-l-K}$, we have $|\EE_{h_i}|<q^{-l+O_K(1)}|\Omega|$, and
$|\EE_{h_1}\cup\EE_{h_2}|<2q^{-l+O_K(1)}|\Omega|<|\Omega|$ if $l$ is sufficiently large.
Then there is some $A_0\in\Omega\setminus(\EE_{h_1}\cup\EE_{h_2})$.
\end{proof}

The next lemma will be used in the proof of Theorem~\ref{thm:7}.

\begin{lemma}\label{lem:16}
Let $N=|\Omega|$. If $\AA,\BB\subset\Omega$ are $2$-cross union, then
$q^{2l}|\AA||\BB|\leq(N-|\AA|)(N-|\BB|)$.
Moreover, if $|\AA|=|\BB|=\qbinom{2l-1}l$, then $\AA=\BB$.
\end{lemma}

\begin{proof}
Let $T$ be the adjacency matrix of $K_q(2l,l)$. Recall that 
$Tv_{i,j}=\theta_iv_{i,j}$ for $i=0,1,\ldots,l$ and $j\in[m_i]$.
Let $f=\one_\AA$, and write $f=\mu_\AA\one+f_*$, where $\mu_\AA=\E f$ and $\langle f_*,\one\rangle=0$.
Then, $|\AA|=\langle f,f\rangle=\mu_\AA^2N+|f_*|_2^2$,
and so $|f_*|_2^2=|\AA|-|\AA|^2/N$.
In the same way, for $g=\one_\BB=\mu_\BB\one+g_*$, we have $|g_*|_2^2=|\BB|-|\BB|^2/N$.
By 2-cross union condition, 
$0=\langle f,Tg\rangle=\langle\mu_\AA\one+f_*,\mu_\BB\theta_0\one+Tg_*\rangle=
\theta_0\mu_\AA\mu_\BB N+\langle f_*,Tg_*\rangle$.
Thus, using \eqref{eq:9} we have
\begin{align}\label{eq:11}
 q^{l^2}\frac{|\AA||\BB|}N=\theta_0\mu_\AA\mu_\BB N
=-\langle f_*,Tg_* \rangle\leq|\theta_1| |f_*|_2 |g_*|_2
=q^{l^2-l}\sqrt{|\AA|-\frac{|\AA|^2}N}\sqrt{|\BB|-\frac{|\BB|^2}N}.
\end{align}
Rewriting this yields the desired inequality.

Next, we further assume that $|\AA|=|\BB|=M:=\qbinom{2l-1}l$, and so $|f|_2=|g|_2$,
$|f_*|_2=|g_*|_2$.
In this case all the terms in \eqref{eq:11} coincide, and 
$\langle f_*,Tg_* \rangle= \theta_1|f_*|_2 |g_*|_2$.
Since $|\theta_1|>|\theta_i|$ for $i\geq 2$, we have $Tg_*=\theta_1g_*$.
This yields $\langle f_*,\theta_1 g_*\rangle=\theta_1|f_*|_2|g_*|_2$, and
$\langle f_*,g_*\rangle=|f_*|_2|g_*|_2$. Thus it follows from equality in the
Cauchy-Schwarz inequality that $f_*$ is a positive multiple of $g_*$.
Moreover, since $|f_*|_2=|g_*|_2$, we have $f_*=g_*$. 
Then, since $|f|_2=|g|_2$, we get $\mu_\AA=\mu_\BB$, 
and finally, $f=g$, or equivalently, $\AA=\BB$.
\end{proof}

\subsection{Large cross union subfamilies in a quotient space}
\label{sec:degenerate subsp}
Let $V=\F_q^{ml}$ and $\Omega=\qbinom Vl$.
Let $D\in\Omega$ and $\ol V=V/D$.
For $\FF\subset\Omega$ and a quotient map
$\pi_D:V\to\ol V$,
we define
\begin{align}\label{eq:12}
\FF(D):=\{F\in\FF:F\cap D=\{0\}\},\quad
\ol\FF^D:=\{\pi_D(F):F\in\FF(D)\}.
\end{align}
Recall that $\theta_0=q^{l^2}\qbinom{(m-1)l}l$ is the largest eigenvalue of $K_q(ml,l)$.
\begin{lemma}\label{lem:17}
Fix $K>0$ and $m\geq 3$.
Let $\FF_1,\ldots,\FF_m\subset\Omega$ be $m$-cross union families.
Suppose that for each $i\in[m]$, 
$\mu_i:=|\FF_i|/|\Omega|\geq q^{-\delta_il-c}$, where $\delta_i\geq 0$, 
$\sum_{i\in[m]}\delta_i\leq m$, and $c=c(m,K)$ is a fixed constant. 
Then, for every $s\in[m]$, there exists $\GG_s\subset\FF_s$ with
$|\GG_s|=(1-o(1))|\FF_s|$ such that 
$|\FF_i(D)|\geq \frac12\theta_0\mu_i$ and
$|\ol\FF_i^D|\geq q^{-l^2}|\FF_i(D)|$ 
for all $D\in\GG_s$ and all $i\in [m]\setminus\{s\}$.
Moreover, the families $\ol\FF_i^D\subset\qbinom{\ol V}l$, where
$i\in [m]\setminus\{s\}$, are $(m-1)$-cross union.
\end{lemma}

\begin{proof}
By renumbering if necessary, we may assume that $s=m$.
For $i\in[m-1]$, let $d_i(D):=|\FF_i(D)|$.
Then, by Lemma~\ref{lem:14} with $\eta=\frac12$, we have
\begin{align*}
&|\{D\in\Omega:d_i(D)<\frac12\theta_0\mu_i\}|
\leq\mu_i^{-1}q^{-2(m-1)l+O_{m,K}(1)}|\Omega|
=\mu_i^{-1}q^{-2(m-1)l+O_{m,K}(1)}\mu_m^{-1}|\FF_m|\\
&\qquad\qquad\leq q^{-(2m-2-\delta_i-\delta_m)l+O_{m,K}(1)}|\FF_m|
\leq q^{-(m-2)l+O_{m,K}(1)}|\FF_m|=o(|\FF_m|).
\end{align*}
Thus $\GG_m:=\bigcap_{i\in[m-1]}\{D\in\FF_m:d_i(D)\geq\frac12\theta_0\mu_i\}$
has size $(1-o(1))|\FF_m|$.
Then, for $D\in\GG_m$ and $i\in[m-1]$, 
it follows that $d_i(D)\geq\frac12\theta_0\mu_i$.
By \eqref{eq:5} every $l$-subspace of $\ol V$ has $q^{l^2}$ lifts,
and $|\ol\FF_i^D|\geq q^{-l^2}d_i(D)$.

If $\ol\FF_i^D$ ($i\in[m-1]$) are not $(m-1)$-cross union, then we have
$\ol F_i=\pi_D(F_i)\in\ol\FF_i^D$ such that $\ol F_1+\cdots+\ol F_{m-1}=\ol V$.
Then, we have $F_1+\cdots+F_{m-1}+D=V$, contradicting the $m$-cross union property. Therefore the families $\ol\FF_i^D$ ($i\in[m-1]$) are $(m-1)$-cross union.
\end{proof}

\subsection{EKL junta theorem and its application}
Let $X,Y$ be vector spaces over $\F_q$ with $n=\dim X=\dim Y$.
We say that $\HH\subset\Hom(X,Y)$ is intersecting if 
$\dim\{x\in X:Mx=Nx\}\ge1$ for all $M,N\in\HH$.
We say that $\HH_1,\HH_2\subset\Hom(X,Y)$ are cross intersecting if 
$\dim\{x\in X:Mx=Nx\}\ge1$, or equivalently, $\ker(M-N)\neq\{0\}$, for all $M\in\HH_1,N\in\HH_2$.

A cell $\CC=\CC(S,T;\Pi,\pi)$ in $\Hom(X,Y)$ is the set of the form
\[
 \CC(S,T;\Pi,\pi)=\{f\in\Hom(X,Y):f|_S=\Pi,\,f^*|_T=\pi\},
\]
where $S\leq X$, $T\leq Y^*$, $\Pi:S\to Y$, and $\pi:T\to X^*$.
If $\CC$ is a non-empty cell, then $|\CC|=q^{(n-\dim S)(n-\dim T)}$.
The complexity of $\CC$ is defined by $\dim S+\dim T$. 
A complexity 0 cell is the whole space $\Hom(X,Y)$.
An $s$-bounded junta is a
union of at most $s$ cells, each of complexity at most $s$. By $s$-bounded, we
always mean that $s$ is a constant depending only on the fixed junta parameter and
$q$ (the order of the field).

We use the following version of a result obtained by Ellis, Kindler, and Lifshitz in \cite{EKL}.

\begin{lemma}\label{lem:18}
For every junta parameter $R\geq 1$ there are constants $C_R=C_R(q)$ and $s_R=s_R(q)$
such that the following holds for all sufficiently large $m$. Let $X,Y$ be $m$-dimensional vector spaces and let
$\HH\subset\Hom(X,Y)$ be intersecting.
Then there is an intersecting $s_R$-bounded junta $J\subset\Hom(X,Y)$ such that
\[
  |\HH\setminus J|
  \leq C_Rq^{-Rm}|\Hom(X,Y)|.
\]
\end{lemma}
This is Theorem~2 of \cite{EKL} specialized to $t=1$ and $r=R$. An intersecting family
is $0$-intersection free, and the dimension condition is automatic because 
$\dim X=\dim Y=m$. The theorem gives a strongly $1$-intersecting $(C_R,R)$-junta,
which is in particular intersecting. Taking $s_R=\max\{C_R,R\}$ gives the formulation
above.

\begin{lemma}\label{lem:19}
 Fix $s>0$. The number of $s$-bounded juntas in $\Hom(X,Y)$ is at most
$q^{O_{q,s}(n)}$.
\end{lemma}
\begin{proof}
 We first count the number of cells $\CC(S,T;\Pi,\pi)\subset\Hom(X,Y)$
with $\dim S=a$, $\dim T=b$, and $a+b\leq s$.
We have $\qbinom na$ choices for $S$, and $\qbinom nb$ choices for $T$.
For fixed $(S,T)$, we have $q^{an}$ choices for $\Pi$, and
$q^{b(n-a)}$ choices for $\pi$, where we used $\pi(t)(x)=t(\Pi x)$ 
for $x\in S$, $t\in T$.
Thus the number of complexity at most $s$
cells is at most $N:=\sum_{a,b\geq 0,\,a+b\leq s}\qbinom na\qbinom nbq^{an}q^{b(n-a)}\leq q^{O_{q,s}(n)}$. An $s$-bounded junta is a union of at most $s$ 
such cells, and the total number is at most 
$\sum_{i=0}^s\binom Ni\leq q^{O_{q,s}(n)}$.
\end{proof}

The next proposition is one of the main tools in this paper.

\begin{proposition}\label{prop:20}
Fix $K>0$. There exist $L_0=K+O(1)$ and $n_0(K,q)$ such that the following holds for all 
$n\geq n_0(K,q)$. Let $A,B$ be $n$-dimensional vector spaces and let 
$f,g:\Hom(B,A)\to[0,1]$ satisfy $\E f,\E g\geq q^{-n-K}$.
Suppose also that $\ker(M-N)\neq\{0\}$ whenever $f(M)g(N)>0$.
Then one of the following items holds.
\begin{enumerate}
\item[\rm (i)] There are $0\ne u\in B$ and $v\in A$ such that
\[
\E[f\mid Mu=v]\geq q^{-L_0}, \quad
\E[g\mid Nu=v]\geq q^{-L_0}.
\]
\item[\rm (ii)] There are $0\ne\phi\in A^*$ and $\psi\in B^*$ such that
\[
\E[f\mid \phi\circ M=\psi]\geq q^{-L_0}, \quad
\E[g\mid \phi\circ N=\psi]\geq q^{-L_0}.
\]
\end{enumerate}
\end{proposition}

The $O(1)$ term in $L_0$ is actually $4+\lceil\log_q(2s_4(q))\rceil$, where $s_4(q)$
comes from Lemma~\ref{lem:18}, and we only need that this term is a constant
depending only on $q$.

The expectation in the above proposition is with respect to the uniform measure
on $\MM:=\Hom(B,A)$, that is, $\E f:=|\MM|^{-1}\sum_{M\in\MM}f(M)$.
Similarly, the conditional expectation is defined by
$\E[f\mid Mu=v]:=|\MM_{u,v}|^{-1}\sum_{M\in\MM_{u,v}}f(M)$, where
$\MM_{u,v}:=\{M\in\MM:Mu=v\}$. For two maps $f_1,f_2:\MM\to\R$, let
$\E[f_1f_2]:=|\MM|^{-1}\sum_{M\in\MM}f_1(M)f_2(M)$.

We will use Proposition~\ref{prop:20} to prove Proposition~\ref{prop:10},
see the paragraph preceding Claim~\ref{claim:30}.
We include a probabilistic proof of Proposition~\ref{prop:20} in Appendix.

\subsection{Two more technical lemmas}
Let $V$ be a $3l$-dimensional vector space over $\F_q$.

\begin{lemma}\label{lem:21}
For a given constant $K_0>0$ and for all sufficiently large $l$, the following holds. Let $\AA\subset\qbinom Vl$ satisfy
$ |\AA|\geq q^{-K_0}\qbinom{3l-1}{l}$, and let $\RR\subset\qbinom V{l+1}$.
If every $R\in\RR$ satisfies 
$|\{A\in\AA:A\cap R\ne0\}|  \geq q^{-K_0}|\AA|$, then $|\RR|\leq q^{O_{K_0}(l)}$.
\end{lemma}

\begin{proof}
For $R,R'\in\qbinom{V}{l+1}$, we define the distance between them by 
$d(R,R'):=(l+1)-\dim(R\cap R')$.
We first generously estimate how many $l$-subspaces meet both $R$ and $R'$ when $d(R,R')$ is large. Let $t=d(R,R')$ and $S=R\cap R'$, so $\dim S=l+1-t$.

If $A\in\qbinom Vl$ satisfies $A\cap S\ne0$, then $A$ contains a line in $S$. Hence the number of such $A$ is at most
$  \qbinom{l+1-t}{1}\qbinom{3l-1}{l-1}
  \leq q^{-t+O(1)}\qbinom{3l-1}{l}$.
If $A\cap S=\{0\}$ but $A$ meets both $R$ and $R'$, then $A$ contains a line $L\leq R$ and a line $L'\leq R'$ whose span is two-dimensional. Thus the number of such $A$ is at most
$ \qbinom{l+1}{1}^2\qbinom{3l-2}{l-2}
  \leq q^{-l+O(1)}\qbinom{3l-1}{l}$.
Choose a constant $t_0=t_0(K_0)$ so large that 
if $t>t_0$, then $q^{-t+O(1)}+q^{-l+O(1)}<\frac12 q^{-3K_0}$. Then,
$|\{A\in\AA:A\cap R\neq 0,\,  \ A\cap R'\neq 0\}|
\leq\frac12 q^{-3K_0} \qbinom{3l-1}l\leq\frac12 q^{-2K_0}|\AA|$ provided $d(R,R')>t_0$.

Now choose a maximal subfamily $R_1,\dots,R_m$ of $\RR$ with pairwise distance greater than $t_0$. Let $X_i=\{A\in\AA:A\cap R_i\ne0\}$.
Then $q^{-K_0}|\AA|\leq |X_i|\leq |\AA|$, and 
$|X_i\cap X_j|<\frac12 q^{-2K_0}|\AA|$ for $i\ne j$.
Let $I(A)=|\{i\in[m]:A\in X_i\}|$, and let $\one_{X}$ denote the indicator vector of $X$.
Then,
\[
I(A)=\sum_{i=1}^m \one_{X_i}(A),\quad
\sum_{A\in\AA}\sum_{i=1}^m \one_{X_i}(A)=\sum_{i=1}^m|X_i|, \quad
\sum_{A\in\AA}\sum\nolimits'\one_{X_i}(A)\one_{X_j}(A)=\sum\nolimits'|X_i\cap X_j|,
\]
where $\sum\nolimits'$ denotes the sum over all $1\leq i<j\leq m$, and
\[
\sum_{A\in\AA} I(A)^2
=\sum_{A\in\AA}\left(\sum_{i=1}^m \one_{X_i}(A)\right)^2
=\sum_{i=1}^m |X_i|+2\sum\nolimits'|X_i\cap X_j|.
\]
Also, we have $mq^{-K_0}|\AA|\leq \sum_{A\in\AA}I(A)\leq m|\AA|$, and  
$\sum\nolimits'|X_i\cap X_j|<\binom m2\left(\frac12q^{-2K_0}|\AA|\right)$.
Thus, by the Cauchy--Schwarz inequality, it follows that
\[
(m q^{-K_0}|\AA|)^2\leq
  \left(\sum_{A\in\AA}I(A)\right)^2
  \leq |\AA|\sum_{A\in\AA}I(A)^2
\leq|\AA|\left(m|\AA|+\frac12m^2 q^{-2K_0}|\AA|\right),
\]
which yields $m\leq 2q^{2K_0}=O_{K_0}(1)$.

Let $\Ball_{j}(R_i):=\{R\in\qbinom V{l+1}:d(R,R_i)\leq j\}$ denote the 
Grassmann ball of radius $j$ centered at $R_i$. 
For a fixed $R_i$ and a constant $j\leq t_0(K_0)$, the number of subspaces 
$R\in\qbinom V{l+1}$ with $d(R,R_i)=j$ is, by \eqref{eq:4},
$q^{j^2}\qbinom{2l-1}j\qbinom{l+1}j=q^{3lj+O_{K_0}(1)}=q^{O_{K_0}(l)}$, and 
$|\Ball_{t_0}(R_i)|\leq\sum_{j=0}^{t_0}q^{O_{K_0}(l)}=(t_0+1)q^{O_{K_0}(l)}$.
Since $\RR\subset\bigcup_{i=1}^m\Ball_{t_0}(R_i)$, we have
$|\RR|\leq m(t_0+1)q^{O_{K_0}(l)}=q^{O_{K_0}(l)}$, because both $m$ and $t_0$ are constants 
depending on $K_0$ only.
\end{proof}

\begin{lemma}\label{lem:22}
For every fixed $c>0$ and all sufficiently large $l$, the following holds.
Let $H\leq V$ be a hyperplane and let $C\in\qbinom Vl$ satisfy $C\not\subset H$.
If $\AA,\BB\subset\qbinom Hl$ satisfy
$|\AA|,|\BB|\geq q^{2l^2-l-c}$, then there exist $A\in\AA$ and $B\in\BB$ such that $A+B+C=V$.
\end{lemma}

\begin{proof}
Let $W=C\cap H$ and $\overline H=H/W$. 
Since $C\not\subset H$, we have $\dim W=l-1$ and 
$\dim\overline H=\dim(H)-\dim W=2l$.
Let $\pi:H\to\overline H$ be the quotient map. Then $\ker\pi=W$.
For each $\overline F\in\qbinom{\overline H}l$, 
$\dim(\pi^{-1}(\overline F))=\dim\ol F+\dim W=2l-1$, and $\overline F$ has $q^{l(l-1)}$ lifts by \eqref{eq:5}.

We bound the number $N$ of $l$-subspaces $F\in\qbinom Hl$ with 
$F\cap W\neq 0$.
By choosing a line in $W$ and choosing $l$-subspaces in $H$ containing the 
line, we see that $N\leq\qbinom{l-1}1\qbinom{3l-2}{l-1}=q^{2l^2-2l+O(1)}$.
It follows from this with $|\AA|\geq q^{2l^2-l-O_c(1)}$ that if $l$ is
sufficiently large, then $\AA_0:=\{A\in\AA:A\cap W=\{0\}\}$ satisfies
$|\AA_0|\geq|\AA|-N\geq q^{2l^2-l-O_c(1)}$. Define $\BB_0$ similarly.

Let $\overline\AA=\{\pi(A):A\in\AA_0\}$. Since each $\ol A\in\ol\AA$ has
at most $q^{l(l-1)}$ lifts in $\AA_0$, we have
$|\overline{\AA}| \geq q^{-l(l-1)}|\AA_0|=q^{l^2-O_c(1)}$.
Similarly, $|\overline{\BB}| \geq q^{l^2-O_c(1)}$.
Thus we have
$|\overline\AA| |\overline\BB|\geq q^{2l^2-O_c(1)}$, and this means that
$\overline\AA$ and $\overline\BB$ are not cross intersecting. Indeed,
if they are cross intersecting, then Theorem~\ref{thm:3} yields that 
$|\overline\AA| |\overline\BB|\leq \qbinom{2l-1}{l-1}^2=q^{2l^2-2l+O(1)}$.

Since $\overline\AA$ and $\overline\BB$ are not cross intersecting,
there are $\overline A\in\overline\AA$ and $\overline B\in\overline\BB$
such that $\overline A\cap\overline B=\{0\}$. Let $A\in\AA_0$ and $B\in\BB_0$
be their lifts. Since $\dim \overline H=2l$, we have
$\overline A+\overline B=\overline H$, and $A+B+W=H$.
Moreover, we can write $C=W\oplus\langle u\rangle$ with a vector $u\notin H$, 
and $A+B+C=V$.
\end{proof}

\section{Proofs}
Let $q$ be a fixed prime power, and we treat $q$ as a constant.
\subsection{Proof of Proposition~\ref{prop:10}}

Let $l$ be a sufficiently large integer, and let $V=\F_q^{3l}$
be a $3l$-dimensional vector space over $\F_q$.
We consider the following assertion $\mathsf P_3(K)$ for $K>0$.

\begin{assertion}[$\mathsf P_3(K)$]
Let $\FF_1,\FF_2,\FF_3\subset\qbinom{V}{l}$ be $3$-cross union families.
If at least one of the following two conditions holds,
\begin{itemize}
\item[\rm (i)]
$\|\FF_1\|_l+\|\FF_2\|_l+\|\FF_3\|_l\geq 3(3l-1)$.
\item[\rm (ii)]
$|\FF_i|\geq q^{-K}\qbinom{3l-1}{l}$ 
for all $i=1,2,3$,
\end{itemize}
then there is a hyperplane $H\leq V$ such that
$\FF_i\subset\qbinom{H}{l}$ for all $i=1,2,3$.
\end{assertion}

Note that under either condition, all families are non-empty for sufficiently large $l$.

\begin{proposition}\label{prop:24}
For every fixed $K>0$, the assertion $\mathsf P_3(K)$ holds
for all sufficiently large $l$, depending on $K$.
\end{proposition}

As we will see, Proposition~\ref{prop:10} is an easy consequence of 
Proposition~\ref{prop:24}. For this proof, we only need the condition (i)
of $\mathsf P_3(K)$. However, we need the condition (ii) in the next section,
where we show that $\mathsf P_m(K)$ ($m\geq 3$) holds by induction on $m$.

\begin{proof}[Proof of Proposition~\ref{prop:24}]
Let $\Omega=\qbinom Vl$, and let $x_i=\|\FF_i\|_l$ for $i=1,2,3$. 
Let $G=K_q(3l,l)$ from \S~\ref{sec:3-d q-K}.
Recall that $G$ is $\theta_0$-regular, where $\theta_0=q^{l^2}\qbinom{2l}l=q^{2l^2+O(1)}$. 
If (i) holds, set $\delta_i=3l-x_i$; otherwise (ii) holds, and set $\delta_i=1$.
In both cases, we have
$\delta_1+\delta_2+\delta_3\leq 3$, and
$\mu_i:=|\FF_i|/|\Omega|\geq q^{-\delta_il-O_K(1)}$.
For $C\in\Omega$, let $\FF_i(C)=\{F\in\FF_i:F\cap C=\{0\}\}$ and $d_i(C)=|\FF_i(C)|$.

\begin{claim}\label{claim:25}
Let $\{i_1,i_2,i_3\}=[3]$.
Then, there exists $\GG_{i_3}\subset\FF_{i_3}$ such that 
$|\GG_{i_3}|=(1-o(1))|\FF_{i_3}|$ and
$d_j(C_{i_3})\geq\frac12 \theta_0\mu_j$ for all $C_{i_3}\in\GG_{i_3}$ and $j\in\{i_1,i_2\}$.
In particular, $d_{i_1}(C_{i_3})d_{i_2}(C_{i_3})\geq q^{4l^2-(\delta_{i_1}+\delta_{i_2})l-O_K(1)}$.
\end{claim}

\begin{proof}
For brevity, we show the case $i_1=1,i_2=2$, and $i_3=3$, 
and we write $\GG$ and $C$ for $\GG_3$ and $C_3$.
(For other cases, we permute families and choose fresh $\GG'$ and $C'$ accordingly.)
Apply Lemma~\ref{lem:14} to $\FF_1$ and $\FF_2$ with $\eta=\frac12$. Then,
for $j=1,2$, we have $|\EE_j|\leq 4q^{-4l+O(1)}\mu_j^{-1}|\Omega|$,
where $\EE_j=\{C\in\Omega:d_j(C)<\frac12\theta_0\mu_j\}$.
Using $|\FF_3|=\mu_3|\Omega|$ and $\delta_j+\delta_3\leq 3$, we have
$|\EE_j|\leq 4q^{-4l+O(1)}\mu_j^{-1}\mu_3^{-1}|\FF_3|=q^{-4l+(\delta_j+\delta_3)l+O_K(1)}|\FF_3|\leq q^{-l+O_K(1)}|\FF_3|$. 
Thus, by choosing $l$ sufficiently large, $|\EE_1|+|\EE_2|=o(1)|\FF_3|$.
Let $\GG:=\FF_3\setminus(\EE_1\cup\EE_2)$. Then $|\GG|=(1-o(1))|\FF_3|$,
and for every $C\in\GG$, we have
$d_j(C)\geq\frac12\theta_0\mu_j=q^{2l^2-\delta_jl-O_K(1)}$ for both $j=1,2$.
Thus $d_1(C)d_2(C)\geq q^{4l^2-(\delta_1+\delta_2)l-O_K(1)}$.
\end{proof}

Fix $\GG\subset\FF_3$ in Claim~\ref{claim:25}, and fix $C\in\GG$. 
Let $\overline V=V/C$, and let
\[
\pi_C:V\to \overline V
\]
be the quotient map.
For a subspace $F\leq V$, let $\pi_C(F)=\{\pi_C(x):x\in F\}$,
and for $\FF\subset\qbinom Vl$, let $\pi_C(\FF)=\{\pi_C(F):F\in \FF\}$.

\begin{claim}\label{claim:26}
For every $1\leq i\leq 3$, we have $x_i=3l-1+O_K(1/l)$ and
$|\FF_i|=q^{2l^2-l+O_K(1)}$. In particular, for every $1\leq i\leq 3$,
$|\FF_i|\geq q^{-\tilde K}\qbinom{3l-1}l$, where $\tilde K>0$ is a constant which may depend on $K$.
\end{claim}

\begin{proof}
Let $\overline{\FF_i(C)}=\pi_C(\FF_i(C))\subset\qbinom{\overline V}l$. 
Then $\overline{\FF_1(C)}$ and $\overline{\FF_2(C)}$ are cross intersecting. 
Indeed, if there are
$\overline{F_1}\in\overline{\FF_1(C)}$ and $\overline{F_2}\in\overline{\FF_2(C)}$
with $\overline{F_1}\cap\overline{F_2}=\{0\}$, then there are $F_1\in\FF_1(C)$ and 
$F_2\in\FF_2(C)$ such that $F_1+F_2+C=V$, which contradicts the 3-cross union property.
By Theorem~\ref{thm:3}, we have 
\[
d_1(C) d_2(C)\leq q^{2l^2}|\ol{\FF_1(C)}||\ol{\FF_2(C)}|\leq
q^{2l^2}\qbinom{2l-1}{l-1}^2=q^{4l^2-2l+O(1)}. 
\]
In the same way, we get $d_{i_1}(C_{i_3})d_{i_2}(C_{i_3})\leq q^{4l^2-2l+O(1)}$
for all $\{i_1,i_2,i_3\}=[3]$. By Claim~\ref{claim:25}, it follows that
$d_j(C_i)\geq\frac12\theta_0\mu_j=\frac12q^{l^2}\qbinom{2l}l|\FF_j|/\qbinom{3l}l=q^{-O(1)}|\FF_j|$ for $i\neq j$, and so 
$|\FF_j|\leq d_j(C_i)q^{O(1)}$. Thus, we have 
\begin{align}\label{eq:13}
|\FF_{i_1}||\FF_{i_2}|\leq d_{i_1}(C_{i_3})d_{i_2}(C_{i_3})q^{O(1)}
\leq q^{4l^2-2l+O(1)}. 
\end{align}

First suppose that (ii) holds. 
Then we have $|\FF_i|\geq q^{-K}\qbinom{3l-1}l=q^{2l^2-l-O_K(1)}$ for all $i$. 
By this with \eqref{eq:13},
we obtain $|\FF_i|=q^{2l^2-l+O_K(1)}$, and $x_i=3l-1+O_K(1/l)$.

Next suppose that (i) holds. 
(In this case, $|\FF_i|/|\Omega|\geq q^{-\delta_il-O(1)}$, and 
the error term is independent of $K$.)
By Claim~\ref{claim:25}, $q^{4l^2-(\delta_{i_1}+\delta_{i_2})l-O(1)}\leq d_{i_1}(C_{i_3})d_{i_2}(C_{i_3})$. This together with \eqref{eq:13} yields
\[
 \delta_{i_1}+\delta_{i_2}\geq 2-O(1/l)
\]
for all $i_1\neq i_2$.
Then $\delta_1+3\geq\delta_1+(\delta_1+\delta_2+\delta_3)=
(\delta_1+\delta_2)+(\delta_1+\delta_3)\geq 4-O(1/l)$, and $\delta_1\geq 1-O(1/l)$.
On the other hand, $3\geq \delta_1+(\delta_2+\delta_3)\geq\delta_1+2-O(1/l)$, and
$\delta_1\leq 1+O(1/l)$. 
Hence, $\delta_1=1+O(1/l)$, or equivalently, $x_1=3l-1+O(1/l)$.
Clearly, the same holds for $x_2$ and $x_3$ as well.
Moreover, for every $1\leq i\leq 3$, $|\FF_i|\geq q^{-K_1}\qbinom{3l-1}l$, 
where $K_1$ is a constant independent of $K$ (but may depend on $q$). 

Let $\tilde K=\max\{K,K_1\}$. Then $|\FF_i|\geq q^{-\tilde K}\qbinom{3l-1}l$ if either (i) or (ii) holds.
\end{proof}

Each $\overline F\in\qbinom{\overline V}l$ has $q^{l^2}$ lifts
$F\in\qbinom Vl$ such that $\pi_C(F)=\overline F$, and in this case $F\cap C=\{0\}$. 
So, for $i=1,2$, we can define a map $f_i:\qbinom{\overline V}l\to[0,1]$ by
\[
f_i(\overline F):=q^{-l^2}|\{F\in\FF_i:\pi_C(F)=\overline F\}|. 
\]
Then, $q^{l^2}\sum_{\overline F\in\qbinom{\overline V}l}f_i(\overline F)=
|\{F\in\FF_i:F\cap C=\{0\}\}|=d_i(C)\geq\frac12\theta_0\mu_i$.
Thus,
\[
 \E f_i=\qbinom{2l}l^{-1}\sum_{\overline F\in\qbinom{\overline V}l}f_i(\overline F)
\geq \qbinom{2l}l^{-1}q^{-l^2}\frac12\theta_0\mu_i=\frac12\mu_i
\geq q^{-l-O_K(1)}.
\]
\begin{claim}
If $f_1(\overline F_1)f_2(\overline F_2)>0$ then $\overline F_1\cap \overline F_2\neq 0$.
\end{claim}

\begin{proof}
Suppose that $f_1(\overline F_1)f_2(\overline F_2)>0$. Let
$F_i\in\FF_i$ be a lift of $\ol F_i$ for $i=1,2$.
If $\overline F_1\cap \overline F_2=\{0\}$, then $\overline F_1+\overline F_2=\overline V$, and $(F_1+F_2+C)/C=\overline V$. But this means $F_1+F_2+C=V$,
contradicting the 3-cross union property.
\end{proof}

Since $\E f_i\geq q^{-l-O_K(1)}$, we can apply Lemma~\ref{lem:15} to $f_i$ ($i=1,2$). 
Then there exists $\overline A_0\in\qbinom{\overline V}l$ such that 
for both $i=1,2$,
\begin{align}\label{eq:14}
\E[f_i\mid \overline F\cap \overline A_0=\{0\}]=
 q^{-l^2}\sum_{\overline F\cap\overline A_0=\{0\}}f_i(\overline F)\geq q^{-l-O_K(1)}.
\end{align}
Choose any complement $\overline B_0\in\qbinom {\overline V}l$ so that $\overline V=\overline A_0\oplus\overline B_0$.
For a linear map $M\in\Hom(\overline B_0,\overline A_0)$, we define its graph $\Gamma(M)$ by
\[
\Gamma(M):=\{(Mb,b)\in\overline A_0\oplus\overline B_0:b\in\overline B_0\}. 
\]

\begin{claim}\label{claim:28}
There is a bijection between $\Hom(\overline B_0,\overline A_0)$ and 
$\{\overline F\in\qbinom{\overline V}l:\overline F\cap \overline A_0=\{0\}\}$
by sending $M$ to $\Gamma(M)$.
\end{claim}

\begin{proof}
Let $M\in\Hom(\overline B_0,\overline A_0)$. First, we show that 
if $\overline F:=\Gamma(M)\in\qbinom{\overline V}l$, then
$\overline F\cap \overline A_0=\{0\}$. Indeed, if $f\in\overline F$ and
$f=(Mb,b)\in \ol A_0$, then $b=0$ and so $Mb=0$, that is $f=0$.
Next, we show that for every $\overline F\in\qbinom{\overline V}l$ with
$\overline F\cap\overline A_0=\{0\}$,
there exists $M$ such that $\overline F=\Gamma(M)$.
By restricting the projection $\overline A_0\oplus\overline B_0\to\overline B_0$ to $\overline F$, we get an isomorphism $\overline F\to\overline B_0$.
Thus, for every $b\in\overline B_0$, there is a unique $a\in\overline A_0$
such that $(a,b)\in\overline F$. This determines $M\in\Hom(\overline B_0,\overline A_0)$ by $Mb=a$.
\end{proof}

For $i=1,2$, define $g_i:\Hom(\overline B_0,\overline A_0)\to[0,1]$ by
\[
g_i(M):=f_i(\Gamma(M)).  
\]
Then, by Claim~\ref{claim:28} with \eqref{eq:14}, we have
\begin{align*}
\E g_i&=|\Hom(\overline B_0,\overline A_0)|^{-1}\sum_M g_i(M)
=q^{-l^2}\sum_M f_i(\Gamma(M))
=q^{-l^2}\sum_{\overline F\cap\overline A_0=\{0\}}f_i(\overline F)
\geq q^{-l-O_K(1)}.
\end{align*}

\begin{claim}
If $g_1(M_1)g_2(M_2)>0$ then $\ker(M_1-M_2)\neq \{0\}$.
\end{claim}

\begin{proof}
Suppose that $g_1(M_1)g_2(M_2)>0$, that is, $f_1(\Gamma(M_1))f_2(\Gamma(M_2))>0$. Then, $\Gamma(M_1)\cap\Gamma(M_2)\neq 0$, and there is some
non-zero $(a,b)\in\Gamma(M_1)\cap\Gamma(M_2)$. In this case,
$(M_1b,b)=(M_2b,b)$, or equivalently $(M_1-M_2)b=0$, 
where $0\neq b\in\overline B_0$. Thus, $\ker(M_1-M_2)\neq \{0\}$.
\end{proof}

For each $C\in\GG\subset\FF_3$ in Claim~\ref{claim:25} 
and $\ol V=V/C=\ol A_0\oplus \ol B_0$, 
we apply Proposition~\ref{prop:20} 
with $A:=\overline A_0$, $B:=\overline B_0$, $f:=g_1$, $g:=g_2$, and $n:=l$. 
Then one of the following holds.

\begin{enumerate}
\item[(i)] There are $0\ne u\in \overline B_0$ and $v\in \overline A_0$ such that
$\E[g_i\mid Mu=v]\geq q^{-O_K(1)}$ for $i=1,2$.
\item[(ii)] There are $0\ne\phi\in \overline A_0^*$ and $\psi\in \overline B_0^*$ such that $\E[g_i\mid \phi\circ M=\psi]\geq q^{-O_K(1)}$ for $i=1,2$.
\end{enumerate}

\begin{claim}\label{claim:30}
There exists some $C\in\GG$ satisfying (ii). 
\end{claim}
\begin{proof}
Suppose, to the contrary, that there is no $C\in\GG$ satisfying (ii).
Then, for each $C\in\GG$, only (i) holds and there are $0\neq u_C\in \ol B_0$ 
and $v_C\in \ol A_0$ such that $\E[g_i\mid Mu_C=v_C]\geq q^{-O_K(1)}$ for $i=1,2$.
Consider a line $\ol L_C=\langle (v_C,u_C)\rangle\leq\ol V$.
Let $L_C=\pi_C^{-1}(\ol L_C)$.
Then $\ker\pi_C=C\subset L_C$ and $\dim L_C=\dim C+\dim \ol L_C=l+1$.

Let $\MM_C=\{M\in\Hom(\ol B_0,\ol A_0):Mu_C=v_C\}$. 
Then $|\MM_C|=q^{l(l-1)}$ from \eqref{eq:6}.
We have $\sum g_i(M)=\sum f_i(\Gamma(M))=\sum q^{-l^2}|\{F\in\FF_i:\pi_C(F)=\Gamma(M)\}|$, 
where the sum is taken over all $M\in\MM_C$.
If $M\in\MM_C$, then $\overline L_C=\langle (Mu_C,u_C)\rangle\subset\Gamma(M)$.
So, if $\pi_C(F)=\Gamma(M)$, then 
choose $0\ne\ol x\in\ol L_C\subset\Gamma(M)$ and then choose $x\in F$ with
$\pi_C(x)=\ol x$; hence $0\ne x\in F\cap L_C$. 

Let $\RR=\{L_C\in\qbinom{V}{l+1}:C\in\GG\}$ be the family of distinct subspaces.
Recall that $|\FF_i|\geq q^{-\tilde K}\qbinom{3l-1}l$ for some 
$\tilde K=\tilde K(K)$ by Claim~\ref{claim:26}. 
We want to apply Lemma~\ref{lem:21} with $\FF_1$ for $\AA$, and 
$K':=\max\{\tilde K,J\}$ for $K_0$, where $J=J(K)$ will be specified below.
To this end, we claim that $|\{F\in\FF_i:F\cap L_C\neq 0\}|\geq q^{-J} |\FF_i|$ 
for some $J=J(K)$.  
Indeed, we have
\begin{align*}
|\{F\in\FF_i:F\cap L_C\neq 0\}|
&\geq \sum_{M\in\MM_C} |\{F\in\FF_i:\pi_C(F)=\Gamma(M)\}|
=q^{l^2}\sum_{M\in\MM_C} g_i(M)\\
&=q^{l^2}|\MM_C|\cdot\E[g_i\mid Mu_C=v_C]\geq q^{2l^2-l-O_K(1)}
\geq q^{-J}|\FF_i|, 
\end{align*}
where we used $|\FF_i|=q^{2l^2-l+O_K(1)}$ by Claim~\ref{claim:26} in the last 
inequality.
Hence Lemma~\ref{lem:21} gives $|\RR|\leq q^{O_{K'}(l)}$. 
Since $K'$ depends only on $\tilde K(K)$ and $J(K)$, we have 
$|\RR|\leq q^{O_{K}(l)}$. 
By the map $\GG\ni C\mapsto L_C\in\RR$, each $L\in\RR$ corresponds to 
at most $\qbinom{l+1}{l}=q^{l+O(1)}$ subspaces in $\GG$.
Thus, $|\GG|\leq |\RR|\qbinom{l+1}l\leq q^{O_K(l)}$.
On the other hand, 
$|\GG|\geq(1-o(1))|\FF_3|\geq q^{2l^2-l-O_K(1)}$,
which is impossible for large $l$.
\end{proof}

By Claim~\ref{claim:30}, we can fix $C\in\GG$, 
$\ol V=\ol A_0\oplus \ol B_0$, $\phi$ and $\psi$ satisfying 
the condition (ii).
Let $\overline H_C=\{(a,b)\in \ol A_0\oplus \ol B_0:  \phi(a)=\psi(b)\}$.
Then $\ol H_C$ is a hyperplane in $\ol V$. To see this, let 
$f(a,b):=\phi(a)-\psi(b)$ be a non-zero linear map from $\ol V$ to $\F_q$,
then $\dim\ol H_C=\dim\ker f=\dim\ol V-\dim\F_q=2l-1$.

Suppose that $\phi\circ M=\psi$. We note that
\[
\phi\circ M=\psi\Leftrightarrow
\forall b\in\ol B_0,\,\phi(Mb)=\psi(b)\Leftrightarrow
\forall b\in\ol B_0,\,(Mb,b)\in\ol H_C\Leftrightarrow
\Gamma(M)\subset\ol H_C.
\]
Let $H_C=\pi_C^{-1}(\overline H_C)$.
Then $\ker\pi_C=C\subset H_C$, and $H_C$ is a hyperplane in $V$.
Since $\pi_C(F)=\Gamma(M)$ implies $F\subset H_C$, we have
$\AA_i:=\bigcup_{\phi\circ M=\psi}\{F\in\FF_i:\pi_C(F)=\Gamma(M)\}\subset\FF_i\cap\qbinom{H_C}l$. Then we have
\begin{align*}
 \left|\AA_i\right|&=
\sum_{\phi\circ M=\psi}|\{F\in\FF_i:\pi_C(F)=\Gamma(M)\}|
=q^{l^2}\sum_{\phi\circ M=\psi}f_i(\Gamma(M))\\
&=q^{l^2}\sum_{\phi\circ M=\psi}g_i(M)\geq q^{2l^2-l-O_K(1)},
\end{align*}
where we used \eqref{eq:7} and
$\sum_{\phi\circ M=\psi}g_i(M)=q^{l(l-1)}\E[g_i\mid\phi\circ M=\psi]\geq 
q^{l(l-1)-O_K(1)}$. 

Therefore, all $\AA_1,\AA_2$, and $\FF_3$ are of size at least $q^{2l^2-l-O_K(1)}$.
Choose a fixed constant $c=c(K)$ large enough to dominate all the implicit
$O_K(1)$ terms in the three lower bounds above. 
Then each of $\AA_1,\AA_2$, and $\FF_3$ has size at least $q^{2l^2-l-c}$,
and we may then apply Lemma~\ref{lem:22} to any two of $\AA_1,\AA_2$, and $\FF_3$.
We first apply Lemma~\ref{lem:22} to $\AA_1,\AA_2$.  
After proving $\FF_3\subset\qbinom{H_C}{l}$, we may also
apply it to $\AA_2,\FF_3$ (and symmetrically).
First we show that $\FF_3\subset\qbinom{H_C}l$.
Suppose, to the contrary, that there is $C'\in\FF_3$ with 
$C'\not\subset H_C$. 
By Lemma~\ref{lem:22}, there are $A\in\AA_1$ and $B\in\AA_2$ such that
$A+B+C'=V$. This contradicts the 3-cross union property.

Next we show that $\FF_1\subset\qbinom{H_C}l$.
Suppose, to the contrary, that there is $A'\in\FF_1$ with 
$A'\not\subset H_C$. 
We substitute $(\AA_2, \FF_3, A')$ into $(\AA, \BB, C)$ in Lemma~\ref{lem:22}. 
Then there are $B\in\AA_2$ and $C''\in\FF_3$ such that
$A'+B+C''=V$, contradicting the 3-cross union property.

As in the previous case, we also have $\FF_2\subset\qbinom{H_C}l$,
and consequently, all three families $\FF_1,\FF_2$, and $\FF_3$ are 
contained in the same hyperplane $H_C$.
This completes the proof of Proposition~\ref{prop:24}.
\end{proof}

\begin{proof}[Proof of Proposition~\ref{prop:10}]
Let $\FF_1,\FF_2,\FF_3\subset\qbinom{V_{3l}}l$ be 3-cross union families.
If $\|\FF_1\|_l+\|\FF_2\|_l+\|\FF_3\|_l<3(3l-1)$, then there is nothing to prove.
Suppose that $\|\FF_1\|_l+\|\FF_2\|_l+\|\FF_3\|_l\geq 3(3l-1)$.
Apply Proposition~\ref{prop:24} with $K=1$.
Since (i) of ${\mathsf P_3(1)}$ holds, 
all $\FF_i$ are contained in a common hyperplane $H$, and so 
$|\FF_i|\leq |\qbinom Hl|=\qbinom{3l-1}l$. This and (i) yield
$|\FF_i|=\qbinom{3l-1}l$ for all $i$, and $\FF_1=\FF_2=\FF_3=\qbinom{H}l$. 
Consequently, we have $\|\FF_1\|_l+\|\FF_2\|_l+\|\FF_3\|_l\leq 3(3l-1)$, and 
if equality holds, then $\FF_1=\FF_2=\FF_3=\qbinom{H}l$ for some hyperplane $H$.
Conversely, if $\FF_1=\FF_2=\FF_3=\qbinom{H}l$ for some hyperplane $H$,
then we clearly have $\|\FF_1\|_l+\|\FF_2\|_l+\|\FF_3\|_l=3(3l-1)$.
\end{proof}

\subsection{Proof of Proposition~\ref{prop:11}}
We prove Proposition~\ref{prop:11} by induction on the number of families.
To this end, we need the following slightly stronger, technical assertion
$\mathsf P_m(K)$ for $K>0$. Let $V_{ml}=\F_q^{ml}$.

\begin{assertion}[$\mathsf P_m(K)$]
Let $\FF_1,\ldots,\FF_m\subset\qbinom{V_{ml}}{l}$ be $m$-cross union families.
If at least one of the following two conditions holds,
\begin{itemize}
\item[\rm (i)]
$\sum_{i=1}^m\|\FF_i\|_l\geq m(ml-1)$,
\item[\rm (ii)]
$|\FF_i|\geq q^{-K}\qbinom{ml-1}{l}$ 
for all $i\in[m]$,
\end{itemize}
then there is a hyperplane $H\leq V_{ml}$ such that
$\FF_i\subset\qbinom{H}{l}$ for all $i\in[m]$.
\end{assertion}

Note that once $\mathsf P_m(K)$ is known for one fixed $K$, it follows that
all $m$-cross union families satisfy $\sum_{i=1}^m\|\FF_i\|_l\leq m(ml-1)$.
Indeed, otherwise (i) holds, and then $\FF_i\subset\qbinom Hl$ implies 
$\|\FF_i\|_l\leq ml-1$ for all $i$.
Note also that under either condition, all families are non-empty for sufficiently large $l$. Theorem~\ref{thm:7} shows that $\mathsf P_2(K)$ does not hold.
We prove the following.
\begin{proposition}\label{prop:32}
For every $m\geq 3$ and every fixed $K>0$, the assertion $\mathsf P_m(K)$ holds
for all sufficiently large $l$, depending on $m$ and $K$.
\end{proposition}

\begin{proof}
We prove $\mathsf P_m(K)$ by induction on $m$.  
The initial step $m=3$ follows from Proposition~\ref{prop:24}.
Let $m\geq 4$. Assume that $\mathsf P_{m-1}(K')$ holds for every fixed $K'>0$.
Fix $K>0$, and let
$\FF_1,\ldots,\FF_m\subset\qbinom{V_{ml}}{l}$ be $m$-cross union families.
Suppose that they satisfy either condition~(i) or condition~(ii) in the
definition of $\mathsf P_m(K)$.  Let $V=V_{ml}$, $\Omega=\qbinom{V}{l}$,
$\mu_i=\frac{|\FF_i|}{|\Omega|}$. As in \S~\ref{sec:degenerate subsp},
for $D\in\Omega$, $\FF\subset\Omega$, and a quotient map $\pi_D:V\to\ol V:=V/D$,
we define $\FF(D)$ and $\ol\FF^D$ by \eqref{eq:12}.

\begin{claim}\label{claim:33}
There exist a constant $K_1=K_1(m,K)$ and a subspace $D\in\FF_m$ such that
$\ol\FF_1^D,\ldots,\ol\FF_{m-1}^D\subset\qbinom{\ol V}l$ are $(m-1)$-cross union families with $|\ol\FF_i^D|\geq q^{-K_1}\qbinom{(m-1)l-1}l$
and $|\FF_i(D)|\geq q^{-K_1}\qbinom{ml-1}l$ for all $i\in[m-1]$.
Moreover, $|\FF_m|\geq q^{-K_1}\qbinom{ml-1}l$.
\end{claim}
\begin{proof}
Suppose first that the condition~(i) of $\mathsf P_m(K)$ holds.  
Let $\delta_i:=ml-\|\FF_i\|_l$.
Then, $\delta_i\geq 0$, and (i) is equivalent to 
\begin{align}\label{eq:15}
\delta_1+\cdots+\delta_m\le m. 
\end{align}
By definition, $\mu_i=\qbinom{ml-\delta_i}l/\qbinom{ml}l=q^{-\delta_il+O_m(1)}$.
Let $s\in[m]$ be fixed, and let $M_s:=[m]\setminus\{s\}$. 
Then, by Lemma~\ref{lem:17}, there is $D_s\in\GG_s$ such that 
$\ol\FF_i^{D_s}$  ($i\in M_s$) are $(m-1)$-cross union with 
$|\ol\FF_i^{D_s}|\geq\frac12\qbinom{(m-1)l}l\qbinom{ml-\delta_i}l/\qbinom{ml}l
\geq\frac12\qbinom{(m-1)l-\delta_i}l$,
where the last inequality is equivalent to $q^{(m-1)l-\delta_i-j}(q^l-1)(q^{\delta_i}-1)\geq 0$ for $0\leq j<l$.
Choose a constant $c>0$ so that $q^{-c}\leq 1/2$.
Writing $x=(m-1)l-\delta_i$, we have
\[
\frac{\qbinom{(m-1)l-\delta_i-c/l}l}{\qbinom{(m-1)l-\delta_i}l}=
 \frac{\qbinom{x-c/l}{l}}{\qbinom{x}{l}}
 =q^{-c}\prod_{r=1}^{l}
\frac{1-q^{-(x-l-c/l+r)}}{1-q^{-(x-l+r)}}
 \leq q^{-c}\leq\frac12.
\]
Thus, $|\ol\FF_i^{D_s}|\geq\frac12\qbinom{(m-1)l-\delta_i}l\geq\qbinom{(m-1)l-\delta_i-c/l}l$.
On the other hand, as we noted just before the statement of the proposition,
$\mathsf P_{m-1}(1)$ implies $\sum_{i\in M_s}\|\ol\FF_i^{D_s}\|_l\leq (m-1)((m-1)l-1)$.
Using these two inequalities, we have $\sum_{i\in M_s}\delta_i\geq m-1-O_m(1/l)$.
By this with \eqref{eq:15}, we get $\delta_s\leq 1+O_m(1/l)$.
Since $s$ can be chosen arbitrarily, we have $\delta_i\leq 1+O_m(1/l)$,
 and so $\mu_i\geq q^{-l-O_m(1)}$ for all $i\in[m]$. 
In particular, $|\FF_m|\geq q^{-C_m}\qbinom{ml-1}l$, where $C_m>0$ is a constant
depending on $m$.
Now to apply Lemma~\ref{lem:17} once again, set $\delta_i=1$ for all $i\in[m]$
so that $\sum_{i=1}^m\delta_i=m$ and $\mu_i\geq q^{-l-O_m(1)}$.
Then, by Lemma~\ref{lem:17}, we obtain $D\in\FF_m$ and $(m-1)$-cross union families 
$\ol\FF_1^D,\ldots,\ol\FF_{m-1}^D\subset\qbinom{\ol V}l$ with 
$|\FF_i(D)|\geq q^{-C'_m}\qbinom{ml-1}l$ 
and $|\ol\FF_i^D|\geq q^{-C''_m}\qbinom{(m-1)l-1}l$ for all $i\in[m-1]$.

Suppose next that the condition~(ii) holds. 
In this case we have $\mu_i\geq q^{-l-O_{m,K}(1)}$ for $i\in[m]$.
Then we apply Lemma~\ref{lem:17} with $\delta_i=1$ to get the statement of the claim.
\end{proof}

By Claim~\ref{claim:33}, we can apply (ii) of $\mathsf P_{m-1}(K_1)$ to 
$(m-1)$-cross union families $\ol{\FF}_1^D,\ldots,\ol{\FF}_{m-1}^D$ with 
$|\ol\FF_i^D|\geq q^{-K_1}\qbinom{(m-1)l-1}l$.
Then, there is a hyperplane 
$\ol H\leq\ol V$ such that $\ol\FF_i^D\subset\qbinom{\ol H}l$ for $i\in[m-1]$.
Let $H=\pi_D^{-1}(\ol H)$.  Then $H$ is a hyperplane of $V$ with $D\subset H$,
and $\FF_i(D)\subset\qbinom{H}{l}$ for $i\in[m-1]$.

\begin{claim}\label{claim:34}
For a given constant $K>0$ and all sufficiently large $l$, the following holds.
Let $H\leq V$ be a hyperplane, and let $C\in\qbinom{V}{l}$ satisfy
$C\not\subset H$.  If 
$\AA_1,\ldots,\AA_{m-1}\subset\qbinom{H}{l}$ satisfy 
$|\AA_i|\geq q^{-K}\qbinom{ml-1}{l}$ for all $i$, then
there are $A_i\in\AA_i$ such that
$A_1+\cdots+A_{m-1}+C=V$.
\end{claim}

\begin{proof}
The proof proceeds in exactly the same way as in the proof of 
Lemma~\ref{lem:22}, so we use the same notation and only include a sketch.
Let $W=C\cap H$ and $\ol H=H/W$.
In this case, we have $\dim\ol H=(m-1)l$, $N\leq\qbinom{l-1}1\qbinom{ml-2}{l-1}=q^{(m-1)l^2-(m-1)l+O_{m}(1)}$, and
$|\AA_i|\geq q^{(m-1)l^2-l-O_{m,K}(1)}$.
Then $\AA_{0,i}:=\{A\in\AA_i:A\cap W=\{0\}\}$ has size at least 
$q^{(m-1)l^2-l-O_{m,K}(1)}$, and $\ol \AA_i:=\{\pi(A):A\in\AA_{0,i}\}$
has size at least $q^{-l(l-1)}|\AA_{0,i}|\geq q^{(m-2)l^2-c}$, where $c=O_{m,K}(1)$.
Thus, we have $\|\ol\AA_i\|_l\geq(m-1)l-O_{m,K}(1/l)$, and
$\sum_{i=1}^{m-1}\|\ol\AA_i\|_l>(m-1)\left((m-1)l-1\right)$ for sufficiently large $l$. 
Since the families $\ol\AA_i$ satisfy (i) of $\mathsf P_{m-1}(K)$, if 
they are $(m-1)$-cross union, then they are all contained in the same 
hyperplane, and so $\sum_{i=1}^{m-1}\|\ol\AA_i\|_l\leq(m-1)\left((m-1)l-1\right)$, a contradiction.
Thus the families $\ol\AA_i$ are not $(m-1)$-cross union, 
and there are $\ol A_i\in\ol\AA_i$ such that
$\ol A_1+\cdots+\ol A_{m-1}=\ol H$. Let $A_i$ be a lift of
$\ol A_i$. Then $A_1+\cdots+A_{m-1}+W=H$. 
Finally, let $C=W\oplus\langle u\rangle$ for some vector $u\not\in H$,
and we have $A_1+\cdots+A_{m-1}+C=V$. 
\end{proof}

Recall that $|\FF_i(D)|\geq q^{-K_1}\qbinom{ml-1}{l}$ for $i\in[m-1]$. 
If some $C\in\FF_m$ is not contained in $H$, then we apply
Claim~\ref{claim:34} to the families $\FF_1(D),\ldots,\FF_{m-1}(D)$.
This yields a contradiction to the $m$-cross union property of
$\FF_1(D),\ldots,\FF_{m-1}(D)$, and $\FF_m$.
Hence $\FF_m\subset\qbinom{H}{l}$.

Now fix $j<m$. If some $C\in\FF_j$ is not contained in $H$, 
apply Claim~\ref{claim:34} to the $m-1$ families $\FF_i(D)$ 
$(i\in[m-1]\setminus\{j\})$, 
and $\FF_m$, which satisfies $|\FF_m|\geq q^{-K_1}\qbinom{ml-1}{l}$ by Claim~\ref{claim:33}. 
Then we get a contradiction again. Since $j$ was arbitrary,
we have $\FF_i\subset\qbinom{H}{l}$ for all $i\in[m]$, which is 
the conclusion of $\mathsf P_m(K)$. This completes the inductive proof of
Proposition~\ref{prop:32}.
\end{proof}

\begin{proof}[Proof of Proposition~\ref{prop:11}]
By Proposition~\ref{prop:32}, the assertion
$\mathsf P_r(1)$ holds. Let
$\FF_1,\ldots,\FF_r\subset\qbinom{V_{rl}}{l}$ be $r$-cross union.  If
$\sum_{i=1}^r\|\FF_i\|_l\geq r(rl-1)$,
then (i) of $\mathsf P_r(1)$ gives a hyperplane
$H\leq V_{rl}$ containing every family.  Hence
$\|\FF_i\|_l\leq rl-1$ for all $i$, and therefore
$\sum_{i=1}^r\|\FF_i\|_l\leq r(rl-1)$.

If equality holds, $\mathsf P_r(1)$ gives a common
hyperplane $H$. Since every width is at most $rl-1$ and their sum is
$r(rl-1)$, we have
$\|\FF_1\|_l=\cdots=\|\FF_r\|_l=rl-1$.
Thus each family has the same size as $\qbinom{H}{l}$ and is contained in
it, so
$\FF_1=\cdots=\FF_r=\qbinom{H}{l}$.
Conversely, these common hyperplane families are clearly $r$-cross union
and attain equality.
\end{proof}

\subsection{Proof of Theorem~\ref{thm:7}}

Let $l\geq 2$. First we deal with the case $n=2l$.
For the proof, we use the following version of a corollary of 
Theorem~\ref{thm:1} or Theorem~\ref{thm:3}, which was first proved by Newman~\cite{Newman} and then further extended by Tanaka~\cite{Tanaka} to $t$-union families.

\begin{corollary}[\cite{Newman}]\label{cor:35}
 Let $l\geq 2$. If $\FF\subset\qbinom {V_{2l}}l$ satisfies $F+F'\neq V_{2l}$
for all $F,F'\in\FF$, then $|\FF|\leq\qbinom{2l-1}l$. Moreover, equality holds
if and only if $\FF=\qbinom Hl$ for some $H\in\qbinom {V_{2l}}{2l-1}$, or
$\FF=\{F\in\qbinom {V_{2l}}l:L\leq F\}$ for some $L\in\qbinom {V_{2l}}1$.
\end{corollary}

\begin{proposition}\label{prop:36}
If $\FF_1,\FF_2\subset\qbinom{V_{2l}}l$ are $2$-cross union, then
$\|\FF_1\|_l+\|\FF_2\|_l\leq 2(2l-1)$. 
Moreover, if $\|\FF_1\|_l+\|\FF_2\|_l=2(2l-1)$, then 
$\FF_1=\FF_2=\qbinom Hl$ for some $H\in\qbinom {V_{2l}}{2l-1}$, or
$\FF_1=\FF_2=\{F\in\qbinom {V_{2l}}l:L\leq F\}$ for some $L\in\qbinom {V_{2l}}1$.
\end{proposition}

\begin{proof}
Let $u=\|\FF_1\|_l$, $v=\|\FF_2\|_l$. 
We may assume that $u\leq v$.
If $\FF_2=\qbinom{V_{2l}}l$, then the 2-cross union condition forces $\FF_1=\emptyset$.
If $\FF_1=\emptyset$, then $u+v\leq 0+2l<2(2l-1)$.
Thus, we may assume that $l\leq u\leq v<2l$.

Let $Q:=q^l$, $f(x):=\qbinom xl$, $N:=f(2l)$, and $\phi(x):=\frac{Qf(x)}{f(2l)-f(x)}$
for $l\leq x<2l$ and $\phi(2l):=+\infty$.
Since $f$ is strictly increasing and $z\mapsto Qz/(N-z)$ is strictly increasing on
$[0,N)$, the function $\phi$ is strictly increasing on $[l,2l)$.
By Lemma~\ref{lem:16}, we get
$q^{2l}|\FF_1||\FF_2|\leq(N-|\FF_1|)(N-|\FF_2|)$, that is, 
\begin{align}\label{eq:16}
\phi(u)\phi(v)\leq 1. 
\end{align}

\begin{claim}\label{claim:37}
If $0\leq s<1$, then 
$\phi(2l-1+s)\phi(2l-1-s)\geq 1$ with equality holding if and only if $s=0$.
\end{claim}

\begin{proof}
Since $q^l\qbinom{2l-1}l=\qbinom{2l}l-\qbinom{2l-1}l$ we have $\phi(2l-1)=1$.
Thus, if $s=0$, then the desired inequality holds with equality.

Assume that $s>0$. Let $P_\pm:=(Q+1)f(2l-1\pm s)/f(2l)$.
Then, $\phi(2l-1+s)\phi(2l-1-s)\geq1$ is equivalent to
$P_+\geq\frac{Q+1-P_-}{1+(Q-1)P_-}$. To show this inequality, let
$T_a(t):=\frac{a+1-t}{1+(a-1)t}$. 
Then, for $a\geq Q$ and $0\leq t\leq 1$, 
$T_a(t)-T_Q(t)=\frac{(a-Q)(1-t)^2}{(1+(a-1)t)(1+(Q-1)t)}\geq 0$ with equality
holding iff $a=Q$ or $t=1$. Similarly, we have
\[
T_Q(u)T_Q(v)-T_Q(uv)=\frac{Q(1-u)(1-v)(Q+1+(Q-1)uv)}{(1+(Q-1)u)(1+(Q-1)v)(1+(Q-1)uv)}
\geq 0 
\]
with equality holding iff $u=1$ or $v=1$. 
Let $r_j:=\frac{q^{j-s}-1}{q^j-1}$. 
Noting that $Q+1=\prod_{j=l}^{2l-1}\frac{q^{j+1}-1}{q^j-1}$ and
$P_+=\prod_{j=1}^{2l-1}\frac{q^{j+s}-1}{q^j-1}=\prod_{j=1}^{2l-1}T_{q^j}(r_j)$,
we have
\begin{align*}
P_+=\prod_{j=l}^{2l-1}T_{q^j}(r_j)\geq
\prod_{j=l}^{2l-1}T_{Q}(r_j)\geq
T_{Q}\big(\prod_{j=l}^{2l-1}r_j\big)=T_Q(P_{-})=
\frac{Q+1-P_{-}}{1+(Q-1)P_{-}}.
\end{align*}
The first inequality is strict because $q^{2l-1}\neq Q$ and $r_j<1$, where we
used $l\geq 2$ and $s<1$.
\end{proof}

\begin{claim}\label{claim:38}
If $l\leq u\leq v\leq 2l$, and $u+v>2(2l-1)$, then $\phi(u)\phi(v)>1$.
\end{claim}

\begin{proof}
Let $\delta:=\frac{u+v}2-(2l-1)$ and $s:=\frac{v-u}2$.
Then, $v=2l-1+\delta+s$ and $u=2l-1+\delta-s$ with $\delta>0$.
Using $u,v\leq 2l$, we get $\delta-1\leq s\leq 1-\delta$, so $0\leq s<1$, which we will
need to apply Claim~\ref{claim:37}.
Recall that $\phi(x)$ is strictly increasing in $x\in[l,2l)$. It follows
\[
 \phi(u) \phi(v)= \phi(2l-1+\delta+s) \phi(2l-1+\delta-s)
>\phi(2l-1+s) \phi(2l-1-s)\geq 1,
\]
where we used Claim~\ref{claim:37} in the last inequality.
\end{proof}
By \eqref{eq:16} and Claim~\ref{claim:38},
we get $\|\FF_1\|_l+\|\FF_2\|_l=u+v\leq 2(2l-1)$.
Moreover, if $u+v=2(2l-1)$, then, 
writing $u=2l-1-s$ and $v=2l-1+s$ with $0\leq s<1$, it follows from
Claim~\ref{claim:37} that $s=0$, that is, $u=v=2l-1$.
Then, by Lemma~\ref{lem:16}, we have $\FF_1=\FF_2$.
Thus the extremal structures are determined by Corollary~\ref{cor:35}.
This completes the proof of Proposition~\ref{prop:36}.
\end{proof}

\begin{proof}[Proof of Theorem~\ref{thm:7}]
We prove Theorem~\ref{thm:7} by induction on $s:=2l-n$.
The initial case $s=0$ is Proposition~\ref{prop:36}.
Suppose that Theorem~\ref{thm:7} is true for the case $2l-n=s$, and
we consider the case $2l-n=s+1$. 

Let $\FF_1,\FF_2\subset\qbinom{V_n}l$ be 2-cross union families with $2l-n=s+1$. 
If one of the families is empty, then $\|\FF_1\|_l+\|\FF_2\|_l\leq n\leq 2(n-1)$.
So we may assume that both families are non-empty.
Then, by Lemma~\ref{lem:12},
there are 2-cross union families $\HH_1,\HH_2\subset\qbinom{V_{n+1}}l$ such
that $\|\HH_i\|_l\geq\|\FF_i\|_l+1$ for $i=1,2$.
Since $2l-(n+1)=s$, it follows from the induction hypothesis that
$\|\HH_1\|_l+\|\HH_2\|_l\leq2((n+1)-1)$. Thus, 
\begin{align}\label{eq:17}
\|\FF_1\|_l+\|\FF_2\|_l\leq\|\HH_1\|_l+\|\HH_2\|_l-2\leq 2(n-1). 
\end{align}

Now suppose that $\|\FF_1\|_l+\|\FF_2\|_l=2(n-1)$, and we determine the 
structure of these families. 
In this case, $\|\HH_1\|_l+\|\HH_2\|_l=2n$ by \eqref{eq:17}.
Thus, by induction hypothesis for the equality case, 
either (a) there is a hyperplane $H_n\in\qbinom{V_{n+1}}n$ such that
$\HH_1=\HH_2=\qbinom{H_n}l$, or (b) $n+1=2l$ and $\HH_1=\HH_2=\{H\in\qbinom {V_{n+1}}l:L\leq H\}$ for some $L\in\qbinom {V_{n+1}}1$. 
For the case (a), by Lemma~\ref{lem:12}, we get 
$\FF_1=\FF_2=\qbinom{H_{n-1}}l$ for some hyperplane $H_{n-1}\in\qbinom {V_n}{n-1}$. 
Finally, we show that the case (b) cannot happen. Let $n=2l-1$.
If $L\leq V_n$, then $\FF_1=\FF_2=\{F\in\qbinom {V_n}l:L\leq F\}$, and we can choose
two subspaces $F_1\in\FF_1$ and $F_2\in\FF_2$ so that $F_1\cap F_2=L$.
But $\dim (F_1+F_2)=\dim F_1+\dim F_2-\dim L=2l-1=n$, contradicting the 
2-cross union assumption.
If $L\not\leq V_n$, then write $V_{n+1}=L\oplus V_n$. In this case, $\FF_1=\FF_2=\emptyset$, and $\|\FF_1\|_l+\|\FF_2\|_l\neq 2(n-1)$.

Conversely, the listed configurations are 2-cross union and attain equality.
\end{proof}

\subsection*{The role of AI in this paper}
AI assistance (GPT-5.5 and 5.6 Pro) was used as a research 
support tool for exploring possible proof strategies and identifying potentially 
relevant mathematical techniques for proving Proposition~\ref{prop:10}.
The authors formulated the problem, developed the width framework, and reduced the
central difficulty to Proposition~\ref{prop:10} before AI tools were used
AI assistance explored alternative approaches, helped expose the common structural
obstruction in the unsuccessful or partial routes, and identified the non-obvious 
connection with the the junta theorem of Ellis, Kindler, and Lifshitz for 
intersecting families of linear maps (Lemma~\ref{lem:18} from \cite{EKL}). 
After this exploratory stage, all mathematical arguments were independently verified, 
and refined by the authors. The final manuscript was written by the authors. 

\subsection*{Acknowledgments}
NT thanks Wataru Kai for helpful discussions.
NT was supported by JSPS KAKENHI Grant Number JP23K03201.

\section*{Appendix}

Here we include the proof of Proposition~\ref{prop:20}.
We use the following version of the Chernoff bounds: if \(X\)
is a sum of independent Bernoulli random variables with mean \(\mu\),
then $\P[X\geq t]\leq \left(\frac{\mathrm e\mu}{t}\right)^t$ for $t\ge\mu$,
and $\P[X\leq(1-\delta)\mu]\leq\exp(-\delta^2\mu/2)$ for $0<\delta<1$, in particular,
$\P[X\leq\mu/2]\leq\mathrm e^{-\mu/8}$.

\begin{proof}[Proof of Proposition~\ref{prop:20}]
Let $K\geq 0$ be given. Let $A,B$ be $n$-dimensional vector spaces over $\F_q$, and
let $\MM=\Hom(B,A)$. Then $|\MM|=q^{n^2}$.
Suppose that maps $f,g:\MM\to[0,1]$ satisfy $\E f,\E g\geq q^{-n-K}$, and if 
$f(M)g(N)>0$ then $\ker(M-N)\neq\{0\}$.
For a family $J\subset\MM$, let $\one_{J}$ denote the indicator of $J$,
and let $\E[f\one_{J}]:=|\MM|^{-1}\sum_{M\in\MM}f(M)\one_{J}(M)$.

\begin{claim}\label{claim:39}
There are constants $s=s(q)$ and $n_1(K,q)$ satisfying the following. 
For all $n\geq n_1(K,q)$, there are $s$-bounded cross intersecting juntas 
$J_f,J_g\subset\Hom(B,A)$ such that 
$\E[f\one_{J_f}]\geq q^{-4}\E f$ and $\E[g\one_{J_g}]\geq q^{-4}\E g$.
\end{claim}

\begin{proof}
Let $y$ and $z$ be new non-zero vectors to construct $(n+1)$-dimensional spaces
$B^+:=B\oplus\langle y\rangle$ and $A^+:=A\oplus\langle z\rangle$. 
Let $\MM^+=\Hom(B^+,A^+)$.
For $c=0,1$, define $\iota_c:\MM\to\MM^+$ by
$\iota_c(M)(b+\lambda y)=Mb+c\lambda z$,
where $M\in\MM$, $b\in B$, and $\lambda\in\F_q$, so
$\iota_0$ is a linear embedding and $\iota_1$ is its affine translate. 
Then, the family
\[
  \{\iota_0(M):f(M)>0\}\cup\{\iota_1(N):g(N)>0\}
\]
is intersecting. 
Indeed, for the same $c$, we have $\iota_c(M)(y)=\iota_c(N)(y)$ for all $M,N\in\MM$,
and for the different $c$, we have $\iota_0(M)(b)=\iota_1(N)(b)$, that is,
$(M-N)b=0$, for $M,N\in\MM$ with $f(M)g(N)>0$ and $b\in\ker(M-N)\setminus\{0\}$.

For each $M\in\MM$, sample $\iota_0(M)$ with probability $f(M)$ independently, and let 
$\HH_0$ be the resulting family. Similarly, we get $\HH_1$ by sampling each
$\iota_1(N)$ with probability $g(N)$, and let $\HH:=\HH_0\cup\HH_1\subset\MM^+$ 
be the resulting random family. Note that $\HH$ is intersecting.

To apply Lemma~\ref{lem:18} with $R=4$ and $m=n+1$ to $\HH$, 
let $\Delta_n:=C_4q^{-4(n+1)}|\MM^+|=C_4q^{n^2-2n-3}$, and let $s:=s_4(q)$.
Then, there is an intersecting $s$-bounded junta $J\subset\MM^+$ such that 
$|\HH\setminus J|\leq\Delta_n$. By choosing $n\geq K+O_q(1)$, we may assume that
$\Delta_n<\frac1{10}q^{n^2-n-K}$.
We need to choose $\HH$ (and $J$) with additional conditions discussed below.

Let $\JJ^+_s$ denote the set of $s$-bounded juntas in $\MM^+$.
Then, $|\JJ^+_s|\leq q^{a n}$ for a constant 
$a=a(q)$ by Lemma~\ref{lem:19}.

Let $F:=\E|\HH_0|=|\MM|\E f\geq q^{n^2-n-K}$.
By the Chernoff bound, 
$\P[|\HH_0|<F/2]\leq\mathrm e^{-F/8}$.
Let $\rho:=3\mathrm e/16$ so that $3\mathrm e q^{-4}\leq\rho<1$.
For a fixed $J\in\JJ^+_s$, let  $X_J:=|\HH_0\cap J|$ be the sum of Bernoulli 
random variables with mean $F_J:=\E X_J=\sum_{\iota_0(M)\in J}f(M)$.
If \(F_J<q^{-4}F\), then $\mathrm e F_J/F<\mathrm e q^{-4}\leq\rho/3$ and 
the Chernoff bound yields
\begin{align*}
  \P\left[X_J\geq F/3\right]\leq
  \left(\frac{\mathrm e F_J}{F/3}\right)^{F/3}<\rho^{F/3}=\mathrm e^{-cF},
\end{align*}
where \(c=-\frac13\log\rho>0\).
Since $F\geq q^{n^2-n-K}$ and since there are only $q^{an}$ candidate juntas,
a union bound shows that if $n\geq n_1(K,q)$, then there is a realization $\HH=\HH_0\cup\HH_1$ for which 
$|\HH_0|\geq F/2$, and $F_J\geq q^{-4}F$ if $X_J\geq F/3$ for every $J\in\JJ^+_s$.
A union bound still works if we moreover require $\HH$ to satisfy
$|\HH_1|\geq G/2$, and $G_J\geq q^{-4}G$ if $Y_J\geq G/3$ for every $J\in\JJ^+_s$,
where $G:=\E|\HH_1|\geq q^{n^2-n-K}$, $Y_J:=|\HH_1\cap J|$ and $G_J:=\E Y_J$.

Fix such a realization $\HH$ and apply Lemma \ref{lem:18}. 
It provides an intersecting junta $J\in\JJ^+_s$ with 
$|\HH\setminus J|\le\Delta_n$. Since $\Delta_n\leq F/10$, it follows that
$X_J=|\HH_0\cap J|\geq|\HH_0|-|\HH\setminus J|\geq F/2-F/10=2F/5>F/3$, and so
$F_J\geq q^{-4}F$. In the same way, $G_J\geq q^{-4}G$.

The inverse image under either $\iota_c$ of a cell in
$\MM^+$ is either empty or a cell in $\MM$ of no larger complexity.  
To see this, write \(p:B^+\to B\) for the natural projection and
\(j:A\to A^+\) for the natural inclusion, so that
$\iota_0(M)=jMp$ and $\iota_0(M)^*=p^*M^*j^*$.
Let $\CC=\CC(S,T;\Pi,\pi)\subset\MM^+$ be a cell.  
If \(\iota_0^{-1}(\CC)\) is non-empty, choose
\(M_0\in\iota_0^{-1}(\CC)\), and put
$S_0=p(S)$ and $T_0=j^*(T)$.
Then
$\iota_0^{-1}(\CC)  =  \CC\bigl(S_0,T_0;M_0|_{S_0},M_0^*|_{T_0}\bigr)$.
Moreover, $\dim S_0+\dim T_0\leq \dim S+\dim T$.
Since inverse images commute with unions, 
$J_f:=\iota_0^{-1}(J)=\{M\in\MM:\iota_0(M)\in J\}$
is again an \(s\)-bounded junta if $J\in\JJ^+_s$.  
We also have $\E[f\one_{J_f}]=|\MM|^{-1}F_J\geq q^{-4}\E f$.
Since $\iota_1$ is an affine translate of $\iota_0$, we can argue similarly: 
if $J\in\JJ^+_s$ then $J_g:=\iota_1^{-1}(J)$ is an $s$-bounded junta with
$\E[g\one_{J_g}]\geq q^{-4}\E g$.

Finally we check that $J_f$ and $J_g$ are cross intersecting.
Let $M\in J_f$ and $N\in J_g$. Then, 
$\iota_0(M),\iota_1(N)\in J$ agree on some non-zero $b+\lambda y$.
Then, $Mb=Nb+\lambda z$ yields $\lambda=0$, and so $b\neq 0$. Thus $Mb=Nb$ with 
$b\neq 0$, that is, $\ker(M-N)\neq\{0\}$, as needed.
\end{proof}

\begin{claim}\label{claim:40}
For fixed constants $s\ge1$ and $L\geq 0$, there exists $n_2(s,L,q)$ such that, for all 
$n\geq n_2(s,L,q)$, the following holds.
If $J_f,J_g\subset\Hom(B,A)$ are $s$-bounded cross intersecting juntas with
$\E[f\one_{J_f}]\geq q^{-n-L}$ and $\E[g\one_{J_g}]\geq q^{-n-L}$,
then there are complexity $1$ cells $\CC_f\subset J_f$ and
$\CC_g\subset J_g$ such that
$\E[f\mid M\in\CC_f]\geq q^{-L_0}$ and $\E[g\mid N\in\CC_g]  \geq q^{-L_0}$, 
where $L_0=L+\lceil\log_q(2s)\rceil$.
\end{claim}

We note that $s=s_4(q)$ depends on $q$ only (see Lemma~\ref{lem:18}), and so we can 
write $n_2(L,q)$ for $n_2(s,L,q)$.

If $J$ is an $s$-bounded junta and $J=\CC_1\cup\cdots\cup\CC_s$, then the cells
$\CC_i$ are not necessarily disjoint. Thus, for $f:\MM\to[0,1]$, it follows that
$\sum_{M\in\bigcup_i\CC_i}f(M)\leq\sum_i\sum_{M\in\CC_i}f(M)$. In this sense,
$f(M)$ satisfies subadditivity.

\begin{proof}
If $J_f$ contains a complexity 0 cell, then $J_f=\MM$. Choose any $N\in J_g$
and any invertible map $T\in\MM$, and let $M:=N+T\in J_f$. In this case, $T=M-N$
and $\ker(M-N)=\{0\}$, contradicting the assumption that $J_f$ and $J_g$ are
cross intersecting. This excludes a complexity 0 cell from $J_f$, and 
symmetrically from $J_g$.
Thus, every cell in $J_f\cup J_g$ is complexity at least $1$.

Recall that a non-empty cell
$\CC(S,T;\Pi,\pi)$ with $a=\dim S$ and $b=\dim T$ satisfies
$|\CC(S,T;\Pi,\pi)|=q^{(n-a)(n-b)}$.
In particular, a complexity $1$ cell
has size $q^{n(n-1)}$, whereas every cell of complexity at least $2$ has
size at most $q^{(n-1)^2}$.

Write $J_f=\UU_1\cup\UU_2$, where $\UU_1$ is a union of complexity $1$ cells
and $\UU_2$ is a union of complexity at least $2$ cells. The total number of
cells contained in $\UU_1\cup\UU_2$ is at most $s$.
Since $f(M)\leq1$ for all $M\in J_f$, it follows 
$\sum_{M\in J_f}f(M)=|\MM|\,\E[f\one_{J_f}]\geq q^{n^2-n-L}$.
Using the subadditivity, 
we have $\sum_{M\in\UU_2}f(M)\leq sq^{(n-1)^2}<\frac12q^{n^2-n-L}$, where we used
$n\geq n_2(L,q)$ for the second inequality.
Then $\sum_{M\in\UU_1}f(M)>\frac12q^{n^2-n-L}$, and, by pigeonhole, there is a cell
$\CC_f\subset\UU_1$ such that $\sum_{M\in \CC_f}f(M)\geq\frac1{2s}q^{n^2-n-L}$. 
Since $|\CC_f|=q^{n(n-1)}$, we have 
$\E[f\mid M\in\CC_f]=|\CC_f|^{-1}\sum_{M\in \CC_f}f(M)\geq\frac1{2s}q^{-L}\geq q^{-L_0}$. The same argument applies to $J_g$ to get $\CC_g$.
\end{proof}

\begin{claim}\label{claim:41}
Assume $n\ge2$.  If two non-empty complexity $1$ cells
$\CC_f,\CC_g\subset\Hom(B,A)$ are cross intersecting, then either there
are $0\ne u\in B$ and $v\in A$ such that
$\CC_f=\CC_g=\{M:Mu=v\}$,
or there are $0\ne\phi\in A^*$ and $\psi\in B^*$ such that
$\CC_f=\CC_g=\{M:\phi\circ M=\psi\}$.
\end{claim}

\begin{proof}
Every complexity $1$ cell is either a value cell
$\VV(u,v)=\{M:Mu=v\}$ $(0\ne u\in B)$, 
or dual cell
$\DD(\phi,\psi)=\{M:\phi\circ M=\psi\}$ $(0\ne\phi\in A^*)$.
For two cells $\CC_1,\CC_2$, set $\CC_1-\CC_2=\{M-N:M\in\CC_1,\ N\in\CC_2\}$.
They are cross intersecting if and only if $\CC_1-\CC_2$ contains no invertible map.

Consider first two value cells.  If $u,u'$ are linearly independent,
then $\VV(u,v)-\VV(u',v')=\MM$.
Indeed, for any given $T\in\MM$, choose $N\in \VV(u',v')$ with $Nu'=v'$ and 
$Nu=v-Tu$, and put $M:=N+T$. Then, $Mu=v$, that is, $M\in \VV(u,v)$, and so 
$T\in \VV(u,v)-\VV(u',v')$. But $T$ can be
invertible, which contradicts the cross intersecting condition.
Thus, we may assume that $u'=\alpha u$ for some $0\neq \alpha\in\F_q$. Then
\[
 \VV(u,v)-\VV(u',v') =\{M-N:Mu=v,\,N(\alpha u)=v'\} =\{T:Tu=v-\alpha^{-1}v'\}.
\]
This set contains an invertible map if and only if $w:=v-\alpha^{-1}v'\ne0$.  
Indeed, if $w\neq 0$, then let $u,u_2,\ldots,u_n$ be a basis of $B$,
and let $w,w_2,\ldots,w_n$ be a basis of $A$. Then we get an invertible map
$T$ in the set by setting $Tu=w$ and $Tu_i=w_i$ for $2\leq i\leq n$. 
Thus cross intersecting condition forces
$v'=\alpha v$, which means that the two value cells are identical.

Consider next two dual cells.  If $\phi,\phi'$ are linearly
independent, then $\DD(\phi,\psi)-\DD(\phi',\psi')=\MM$.
Indeed, for any given $T\in\MM$, choose $N\in \DD(\phi',\psi')$ with 
$\phi'\circ N=\psi'$ and $\phi\circ N=\psi-\phi\circ T$ (this is possible because
$(\phi,\phi'):A\to\F_q^2$ is surjective),
and put $M:=N+T$. Then, $\phi\circ M=\psi$, that is, $M\in \DD(\phi,\psi)$, and so 
$T\in \DD(\phi,\psi)-\DD(\phi',\psi')$. But $T$ can be
invertible, which contradicts the cross intersecting condition.
Thus, we may assume that $\phi'=\alpha\phi$ for some $0\neq \alpha\in\F_q$. 
Then $\DD(\phi,\psi)-\DD(\phi',\psi')=\{T:\phi\circ T=\psi-\alpha^{-1}\psi'\}$, which contains an invertible $T$ if and only if $\psi-\alpha^{-1}\psi'\neq 0$.
Thus cross intersecting condition forces
$\psi'=\alpha\psi$, which means that the two dual cells are identical.

Finally, consider two mixed cells. In this case, we show that 
$\VV(u,v)-\DD(\phi,\psi)$ contains an invertible map, and so this case cannot happen
by the cross intersecting condition. Let $\alpha:=\phi(v)-\psi(u)\in\F_q$. Then there is
$0\neq a\in A$ such that $\phi(a)=\alpha$. Indeed, if $\alpha\neq 0$, then choose
$y\in A$ with $\phi(y)=1$ and set $a=\alpha y$. If $\alpha=0$, then choose
$0\neq a\in\ker \phi$, here we need $n\geq 2$. 
Next choose an invertible map $T\in\MM$ so that $Tu=a$. 
Then $\phi(Tu)=\phi(v)-\psi(u)$.
We shall show that we can find $M\in\VV(u,v)$ and $N\in\DD(\phi,\psi)$ such that 
$T=M-N$.
Let $w:=v-Tu$ so that $\phi(w)=\phi(v)-\phi(Tu)=\psi(u)$.
Write $B=\langle u\rangle\oplus C$, where $C$ is an $(n-1)$-dimensional complement
subspace. For $\lambda\in\F_q$ and $c\in C$, define $N\in\MM$ by
$N(\lambda u+ c):=\lambda w+\psi(c)y$, where $y\in A$ satisfies $\phi(y)=1$. Then, 
\[
\phi(N(\lambda u+ c))=\phi(\lambda w+\psi(c)y)=\lambda\phi(w)+\psi(c)\phi(y)
=\lambda\psi(u)+\psi(c)=\psi(\lambda u+c),
\]
that is, $\phi\circ N=\psi$, and $N\in\DD(\phi,\psi)$.
Define $M:=N+T$. Then $Mu=Nu+Tu=w+Tu=v$, and $M\in\VV(u,v)$. 
Consequently, $T=M-N\in\VV(u,v)-\DD(\phi,\psi)$ is invertible, as desired.
\end{proof}

Let $L=K+4$ and $n_0(K,q)=\max\{2,n_1(K,q),n_2(L,q)\}$.
Suppose that $n\geq n_0(K,q)$.
By Claim~\ref{claim:39}, there are $s$-bounded cross intersecting
juntas $J_f,J_g\subset\MM$ with $\E[f\one_{J_f}]\geq q^{-4}\E f$ and
$\E[g\one_{J_g}]\geq q^{-4}\E g$. Since $\E f\geq q^{-n-K}$, we get
$\E[f\one_{J_f}]\geq q^{-n-L}$, and similarly, $\E[g\one_{J_g}]\geq q^{-n-L}$.
Then, by Claim~\ref{claim:40}, there are complexity $1$ cells $\CC_f\subset J_f$
and $\CC_g\subset J_g$ with $\E[f\mid M\in\CC_f]\geq q^{-L_0}$ and
$\E[g\mid N\in\CC_g]\geq q^{-L_0}$.
Finally, by Claim~\ref{claim:41}, either there are $0\ne u\in B$ and $v\in A$ such that
$\CC_f=\CC_g=\{M:Mu=v\}$, or there are $0\ne\phi\in A^*$ and $\psi\in B^*$ such that
$\CC_f=\CC_g=\{M:\phi\circ M=\psi\}$.
\end{proof}

\end{document}